\documentclass[10pt,letterpaper]{amsart}
\pdfoutput=1	
\usepackage[utf8]{inputenc}

\usepackage{hyperref}

\usepackage{amsmath}	
\usepackage{amssymb}	
\usepackage{amsfonts}
\usepackage{amsthm}				

\usepackage{appendix}				

\usepackage{xcolor}	

\usepackage{soul}				
\sethlcolor{yellow}				

\usepackage{graphicx}			
\usepackage{subcaption}	
\graphicspath{{images/}}

\usepackage{media9}

\usepackage{float}

\usepackage[backend=bibtex,style=numeric,citestyle=numeric,maxbibnames=99,giveninits,doi=false,isbn=false,url=false,eprint=false]{biblatex}
\renewbibmacro{in:}{\ifentrytype{article}{}{\printtext{\bibstring{in}\intitlepunct}}}	
\allowdisplaybreaks

\renewcommand{\vec}[1]{\boldsymbol{#1}}
\let\oldhat\hat
\renewcommand{\hat}[1]{\oldhat{\boldsymbol{#1}}}

\newcommand{\de}{\ensuremath{\partial}}						
\newcommand{\dee}{\ensuremath{\textrm{d}}}

\newcommand{\fdf}[1]{\ensuremath{ \frac{\dee}{\dee #1}}}

\newcommand{\inty}[4]{\ensuremath{ \int_{#1}^{#2} \! #3 \, \dee#4 }}

\newcommand{\field}[1]{\mathbb{#1}}

\newtheorem{assumption}{Assumption}
\newtheorem{remark}{Remark}
\newtheorem{theorem}{Theorem}

\newtheorem{lemma}{Lemma}

\newtheorem{corollary}{Corollary}

\numberwithin{lemma}{section}
\numberwithin{example}{section}
\numberwithin{proposition}{section}
\numberwithin{equation}{section}
\numberwithin{theorem}{section}
\numberwithin{corollary}{section}
\numberwithin{remark}{section}
\numberwithin{definition}{section}
\numberwithin{assumption}{section}
\numberwithin{figure}{section}
\newtheorem*{definition*}{Definition}		
\newtheorem*{assumption*}{Assumption}
\newtheorem*{remark*}{Remark}
\newtheorem*{theorem*}{Theorem}
\newtheorem*{lemma*}{Lemma}
\newtheorem*{proposition*}{Proposition}
\newtheorem*{corollary*}{Corollary}
\newtheorem*{example*}{Example}
\newtheorem*{conjecture*}{Conjecture}

\let\Re\relax
\DeclareMathOperator{\Re}{Re}
\let\Im\relax
\DeclareMathOperator{\Im}{Im}

\title{Defect resonances of truncated crystal structures}

\author{Jianfeng Lu, Jeremy L. Marzuola, and Alexander B. Watson}

\begin{document}

\begin{abstract}
Defects in the atomic structure of crystalline materials may spawn electronic bound states, known as \emph{defect states}, which decay rapidly away from the defect.  
Simplified models of defect states typically assume the defect is surrounded on all sides by an infinite perfectly crystalline material. In reality 
the surrounding structure must be finite, and in certain contexts the structure can be small enough that edge effects are significant. In this work we investigate these edge effects and prove the following result. Suppose that a one-dimensional infinite crystalline material hosting a positive energy defect state is truncated a distance $M$ from the defect. Then, for sufficiently large $M$, there exists a resonance \emph{exponentially close} (in $M$) to the bound state eigenvalue. It follows that the truncated structure hosts a metastable state with an exponentially long lifetime. Our methods allow both the resonance frequency and associated resonant state to be computed to all orders in $e^{-M}$. We expect this result to be of particular interest in the context of photonic crystals, where defect states are used for wave-guiding and structures are relatively small. Finally, under a mild additional assumption we prove that if the defect state has negative energy then the truncated structure hosts a bound state with exponentially-close energy. 
\end{abstract}

\maketitle

\tableofcontents

\section{Introduction}

\label{sec:intro}

In this work we seek to understand solutions of the following model partial differential equation 
\begin{equation} \label{eq:td_schro}
    i \de_t \psi = H \psi \quad H := D_x^2 + V(x), \quad D_x := - i \nabla_x.
\end{equation}
Here $\psi(x,t) : \field{R}^d \times \field{R} \rightarrow \field{C}$ is a complex function known as the wave-function, $H$ is the Hamiltonian, and $V(x)$ is a real function known as the potential. The PDE \eqref{eq:td_schro} arises in two application areas which are relevant to the present work. First, \eqref{eq:td_schro} models the wave-like dynamics of electrons in a $d$-dimensional material in the independent-electron approximation (see, e.g. \cite{ashcroft_mermin}). In this case, $\psi(x,t)$ denotes the wave-function of a single electron, and $V(x)$ denotes the electric potential due to the 
material environment. The other application area is in photonics, where the propagation of electromagnetic waves through media with spatially-varying refractive index can be modeled in the paraxial approximation by \eqref{eq:td_schro} (see, e.g. \cite{2000Agrawal}). In this case, $\psi(x,t)$ denotes the envelope of an electromagnetic wave-packet, and $V(x)$ is related to the refractive index of the medium.

In both of these application areas, there is considerable interest in building \emph{wave-guides}: structures which effectively confine waves to a given region or channel. The simplest way to confine waves is by building a structure which hosts a bound state, as we now briefly review. Suppose $H$ has a bound state, i.e. there exists $E \in \field{R}$ and $\phi(x) \in L^2(\field{R}^d)$ such that
\begin{equation}
    H \phi = E \phi.
\end{equation}
Then the solution of \eqref{eq:td_schro} with initial condition $\psi(x,0) = \phi(x)$ is 
\begin{equation}
    \psi(x,t) = e^{- i E t} \phi(x).
\end{equation}
Since the magnitude of the solution $|\psi(x,t)|^2 = | \phi(x) |^2$ is conserved with respect to time, and since $\phi(x)$ spatially decays, stationary solutions describe waves which remain confined for all $t > 0$.

A common strategy for creating a structure which hosts a bound state is to build a structure which is perfectly periodic except for a defect region. Bound states created by defects in an otherwise periodic structure are known as \emph{defect states}. Many previous works have proved existence of defect states across various models; see e.g. \cite{1976Simon_2,1986DeiftHempel,1993GesztesySimon,1997FigotinKlein,Figotin-Klein:98,Hoefer-Weinstein:11,2011BronskiRapti,duchene2015oscillatory,DVW:15}. For related numerical work, see \cite{2001FigotinGoren,2005Soussi}.
In recent years, ``topologically protected'' edge states, which decay away from the physical edge of a material, or away from extended line defects within a material, have also attracted attention for wave-guiding applications \cite{fefferman_leethorp_weinstein_memoirs,2017FeffermanLee-ThorpWeinstein_2,2018FeffermanWeinstein,2019Drouot_2,2018DrouotFeffermanWeinstein_pre,2019Lee-ThorpWeinsteinZhu,2019DrouotWeinstein,2020Drouot}. For related numerical work, see \cite{2018ThickeWatsonLu}. 

In general these works make the un-physical assumption that the periodic structure surrounding the defect, edge, or line defect, extends infinitely away from the defect. It follows that these models capture the dynamics of real wave-guides only approximately. In this work we investigate the validity of this approximation by investigating how bound states of infinite periodic structures supporting a defect state perturb when the structure is truncated, i.e. when the potential $V(x)$ is set equal to zero outside a bounded region.
Before we can state our result, we need to recall the basic theory of resonances (for more detail, see e.g. \cite{DyatlovZworski}). 

Resonances are poles of the resolvent $(H - z)^{-1}$ viewed as a function of $\lambda := \sqrt{z}$ and meromorphically continued to the lower half of the complex plane ($\Im \lambda < 0$). For our purposes, resonances are important because they control the rate of energy decay of solutions of \eqref{eq:td_schro} from bounded subsets of $\field{R}^d$. Specifically, suppose that $H$ has a resonance $z^*$ in the strip $0 < a \leq \Re z \leq b$ whose imaginary part is closer to the real axis than any other resonance in that strip. Let $\chi_R \in C^\infty_c$ denote a smooth cutoff which equals one on $B(0,R)$ for some $R > 0$, and let $\psi \in C^\infty_c$  denote a cutoff with $\text{supp } \psi = [a,b]$. Then one can prove a bound (roughly stated, see e.g. \cite{2001BurqZworski,DyatlovZworski} for more details)
\begin{equation} \label{eq:resonance_pole_exp}
\begin{split}
    \chi_R e^{- i t H} \chi_R \psi(H) = \; &\chi_R \text{Res}\left( e^{- i t \bullet} (H - \bullet)^{-1} , z^* \right) \chi_R \psi(H)     \\
    &+ \text{terms decaying faster in time.}
\end{split}
\end{equation}
It follows that the imaginary part of $z^*$ controls the rate of decay of solutions of \eqref{eq:td_schro} localized spatially in the ball $B(0,R)$ and spectrally in the interval $[a,b]$.

The main result proved in this work is as follows. We prove that when an infinite structure hosting a defect state with positive energy $E > 0$ is truncated sufficiently far from the defect and $d = 1$, the truncated structure hosts a resonance exponentially close (in the distance from the defect to the truncation) in the complex plane to $E$. Furthermore, our methods allow for precise estimates of the exponentially small corrections to all orders.

As a corollary we prove precise estimates on the rate of decay of waves localized at defects in truncations of infinite one-dimensional structures hosting defect states with positive energy $E > 0$. This result clarifies the validity of the un-physical approximation described above, where wave-trapping by finite structures is studied using models which are periodic away from the defect. For the precise statement of our main result in the simplest setting where reflection symmetry holds, see Section \ref{sec:main_result}. For the statement of our result in the general case where reflection symmetry may be broken, see Section \ref{sec:generalization_noparity}.

We hope that our result will inform the design of novel wave-guiding devices based on defect states. Another motivation and potential application of our result is to the design of efficient lasers. Lasers require the existence of a resonance very close to the real axis. Our work implies that truncations of infinite structures which host defect states with positive energy will lase efficiently if they are truncated sufficiently far from the defect. In recent years ``topological'' lasers, whose associated resonances arise because of edge modes, have been proposed and built \cite{2017St-Jeanetal,2018Zhaoetal,2018Partoetal,2018Hararietal,2018Bandresetal}.

By adapting our methods we can prove two further results. First, we prove that when the defect state has negative energy $E < 0$, under an additional assumption which we expect is generically satisfied (see \eqref{eq:extra} in the statement of Theorem \ref{th:main_theorem_bound}), the truncated structure also hosts a bound state whose eigenvalue is exponentially close to $E$ (see Section \ref{sec:trunc_bdstates}). Second, we prove that when semi-infinite structures hosting negative energy bound states at their edge (edge states) are truncated far from the edge, the truncated structure hosts a bound state whose eigenvalue is exponentially close to that of the original edge state eigenvalue (see Section \ref{sec:gen_edge}). The case where $E = 0$ is not immediately amenable to our methods; see Remark \ref{rem:E=0} for a discussion of this. 

The main ingredients of our proofs are the theory of ODE with periodic coefficients (Floquet theory) \cite{MagnusWinkler,Eastham} combined with a fixed point argument introduced by Dyatlov and Zworski (see Section 2.8.1 of \cite{DyatlovZworski}, where the argument is used to show that bound states of the harmonic oscillator become resonances when the potential is set to zero outside an interval). Although our proof does not generalize easily to $d > 1$ because it relies on ODE theory, we expect our result to also hold in that case and we make a general conjecture to this effect (see Section \ref{sec:conjecture}).

\subsection{Outline of paper}

The outline of the paper is as follows. We will start by presenting and proving our main result in the simplest case, where the structure is symmetric under reflection about the origin (parity symmetric). We will first present our result precisely in Section \ref{sec:main_result}, then present the main ideas of the proof in Sections \ref{sec:fixed} and \ref{sec:asymp}. We provide proofs of key estimates in Sections \ref{sec:bound_de_z_thet} and \ref{sec:bound_de_z_2_thet}. We derive a simplified expression for the leading-order part of the asymptotic expansion of the defect state resonance in Section \ref{sec:Theta}. We will then present our main result precisely in the case where reflection symmetry is broken in Section \ref{sec:generalization_noparity}, presenting the aspects of the proof which differ significantly from the reflection-symmetric case in Section \ref{sec:gen_no_parity}. We will finish by discussing two further results which can be proved using our methods and conjecturing generalizations of our results in Section \ref{sec:generalizations}.

\subsection{Related work} 

In this section we discuss existing literature related to scattering resonances of periodic structures in one dimension. Note that we already discussed the existing literature on bound states of perturbed periodic structures in Section~\ref{sec:intro}.

The scattering theory of discrete and continuum periodic structures on the whole or half line perturbed by local perturbations is well-studied, see e.g. \cite{1987Firsova,2011Korotyaev,2012KorotyaevSchmidt,2002DimassiZerzeri,2004Dimassi,dimassi2011spectral,2011IantchenkoKorotyaev,2012IantchenkoKorotyaev}. Regarding \emph{truncated} periodic structures, Klopp \cite{2012Klopp,2015Trinh,2015Trinh_2,klopp2016} and Trinh \cite{2015Trinh,2015Trinh_2} have studied the resonances of truncated periodic structures on the whole or half-line without defects. Duch\^ene-Vuki\'cevi\'c-Weinstein \cite{DVW:15,duchene2015oscillatory,duchene2014scattering,2011DucheneWeinstein} and Drouot \cite{2018Drouot} have studied the scattering, homogenization, and defect modes of highly oscillatory compactly supported potentials, which are equivalent to truncated periodic structures after a re-scaling. Christiansen \cite{2005Christiansen} has studied the resonances of ``steplike'' potentials which are asymptotically equal to constants with different values at large positive and negative $x$. Wave packet propagation in truncated periodic structures has been studied by Molchanov-Vainberg \cite{molchanov2004slowdown,molchanov2006slowing,molchanov2005slowing}.

In this work, we treat the previously unstudied (as far as we are aware) case of truncated periodic Schr\"odinger operators with a defect.  Our approach is motivated by the study of resonances for the semiclassical resolvent in the case of truncated trapping potentials, see \cite{DyatlovZworski}, Chapter $2.8$.  When the defect is such that a gap eigenstate arises in the infinite periodic model, this problem requires the use of transfer matrix machinery/Floquet theory tools to prove that the truncated model has a resonance very near the real axis at a nearby energy.  Our approach gives strong quantitative bounds on the expected lifespan of these states established a general framework for establishing resonances from defect states in $1$ dimensional models.

\subsection*{Acknowledgements} JL is supported in part by National Science Foundation via grant DMS-1454939. JLM is supported in part by NSF CAREER Grant DMS-1352353 and NSF Applied Math Grant DMS-1909035.  JLM also thanks Duke University and MSRI for hosting him during the outset of this research project. The authors would like to thank Mikael Rechtsmann, Semyon Dyatlov, Alexis Drouot, Dirk Hundertmark and Michael I. Weinstein for stimulating discussions which improved this work. 

\section{Statement of main result when parity symmetry holds} \label{sec:main_result}

We now present our main result in the simplest setting. We consider the one-dimensional continuum Schr\"odinger operator 
\begin{equation} \label{eq:Schro}
    H = D_x^2 + V(x) \quad D_x := - i \de_x,
\end{equation}
where $V$, the potential, is a real function. We start with the following regularity and symmetry assumptions on $V$. 
\begin{assumption}[Regularity and translation symmetry of $V$] \label{as:V_assump}
First, we assume $V$ is smooth: $V(x) \in C^\infty(\field{R})$. Then, we assume that $V$ can be written as
\begin{equation}
    V(x) = V_{\rm per}(x) + V_{\rm def}(x),
\end{equation}
where $V_{\rm per}$ is periodic, i.e. $V_{\rm per}(x + 1) = V_{\rm per}(x)$ for all $x \in \field{R}$, and $V_{\rm def}(x)$ is compactly supported, i.e. there exists $\rho \geq 0$ such that
\begin{equation} \label{eq:eventual_per}
    |x| > \rho \implies V_{\rm def}(x) = 0.
\end{equation}
\end{assumption}

\begin{remark}
The assumption that $V_{\rm per}$ has periodicity $1$ is made without loss of generality up to a possible re-scaling of $x$. 
\end{remark}

\begin{remark}
The eventual periodicity assumption \eqref{eq:eventual_per} is essential to our proof because it allows us to apply Floquet theory. Even so, we expect the result to still hold if this assumption is replaced by a weaker spectral gap condition: see Section~\ref{sec:conjecture}. 
\end{remark}

Now suppose that $H$ has a bound state with positive energy, i.e. that there exist $E > 0$ and $\Phi(x) \in L^2(\field{R})$ such that
\begin{equation} \label{eq:bound}
    H \Phi(x) = E \Phi(x).
\end{equation}
Our aim is to prove that when the structure modeled by $H$ is truncated sufficiently far from $x = 0$, the resulting structure supports a resonance. Specifically, we aim to prove that the operator
\begin{equation} \label{eq:H_trunc}
    H_{\rm trunc} := D_x^2 + V_{\rm trunc}(x),
\end{equation}
where
\begin{equation} \label{eq:V_trunc}
    V_{\rm trunc}(x) = \begin{cases} V(x), & |x| \leq M \\ 0, & |x| > M, \end{cases}
\end{equation}
acting on $L^2(\field{R})$ 
has a resonance $z^* \in \field{C}$ with $\text{Re }z > 0, \text{Im }z < 0$ nearby to $E$ in the complex plane for $M$ sufficiently large (in particular, such that $M > \rho$).

We will present our main result first under the following simplifying symmetry assumption. We emphasize that our techniques do not fundamentally rely on this assumption, although it simplifies the statement of our result and its proof. We discuss how our result generalizes to the case where Assumption \ref{as:parity} doesn't hold in Section \ref{sec:generalization_noparity}.
\begin{assumption}[Parity symmetry of $V$] \label{as:parity}
We assume that $V(x)$ is even, i.e.
\begin{equation} \label{eq:even}
    V(-x) = V(x)
\end{equation}
for all $x \in \field{R}$.
\end{assumption} 

To present our result we require some elementary calculations and some notation. First note that by Floquet theory (we recap the aspects of Floquet theory which are important for us in Section \ref{sec:bound_de_z_thet}), the ODE \eqref{eq:bound} can have at most one decaying solution for $|x| > \rho$, and hence the eigenvalue $E$ is non-degenerate. It then follows from even-ness of the potential $V(x)$ that $\Phi(x)$ is either even or odd, i.e.
\begin{equation}
    \Phi(-x) = \Phi(x) \text{ or } \Phi(-x) = - \Phi(x) \text{ for all $x \in \field{R}$.}
\end{equation}

When $\Phi$ is even, 
define $u_z(x) \in C^\infty([0,M])$ to be the solution of
\begin{equation} \label{eq:eq}
    ( D_x^2 + V(x) - z ) u_z = 0, \quad u_z(0) = 1, u_z'(0) = 0
\end{equation}
for arbitrary $z \in \field{C}$. When $\Phi$ is odd, we make the same definitions except the initial data for \eqref{eq:eq} should be changed to $u_z(0) = 0, u_z'(0) = 1$. Note that in either case $\Phi$ satisfies \eqref{eq:eq} when $z = E$. We define
\begin{equation} \label{eq:X}
    (X_1(z) , X_2(z)) := ( u_z(M) , u_z'(M) ),
\end{equation}
and
\begin{equation} \label{eq:theta}
    \Theta(z) := X_2(z) - i \sqrt{z} X_1(z).
\end{equation}
For the square root in \eqref{eq:theta}, 
we assume $z \in \field{C} \setminus (-\infty,0]$ and choose the branch of the square root such that $\pm \Im z > 0 \implies \pm \Im \sqrt{z} > 0$. With these definitions, $z^* \in \field{C}$ with $\Re z^* > 0$ and $\Im z^* < 0$ is a resonance whenever $\Theta(z^*) = 0$.

The main theorem underlying our results is the following:
\begin{theorem} \label{th:main_theorem}
Let $V(x)$ satisfy Assumptions \ref{as:V_assump} and \ref{as:parity}, and let $\Phi$ be a bound state with positive eigenvalue $E > 0$ as in \eqref{eq:bound}. Let $\rho > 0$ be as in equation \eqref{eq:eventual_per}. Then there exist $k > 0$ and an $M_0 > \rho \geq 0$ such that for all $M \geq M_0$,
\begin{enumerate}
\item $\Theta(z)$ defined by \eqref{eq:theta} has a unique root $z^*$ in the ball 
\begin{equation} \label{eq:Om_M}
    \Omega_M = \left\{ z : |z - E| \leq \frac{1}{M^2} e^{- k M} \right\}.
\end{equation}
\item The location of the root $z^*$ can be precisely characterized as 
\begin{equation} \label{eq:asymp}
    z^* = E - \frac{ \Theta(E) }{ \de_z \Theta(E) } + O(e^{- 4 k M}).
\end{equation}
\end{enumerate}
\end{theorem}
We explain the main ideas of the proof of Theorem \ref{th:main_theorem} in Sections \ref{sec:fixed} and \ref{sec:asymp}, postponing proofs of key estimates to Sections \ref{sec:bound_de_z_thet} and \ref{sec:bound_de_z_2_thet}. The overall idea of the proof is to adapt a fixed point argument introduced by Dyatlov and Zworski, who used it to prove an analogous result to Theorem \ref{th:main_theorem} for the truncated harmonic oscillator in the semiclassical limit (see Chapter 2.8.1 of \cite{DyatlovZworski}). In our setting, largeness of the truncation length $M$ can be understood as playing the role of smallness of the semiclassical parameter in Dyatlov and Zworski's proof. The most difficult step in the proof is to prove a particular map $\Psi : \field{C} \rightarrow \field{C}$ is a contraction in the ball $\Omega_M$ \eqref{eq:Om_M}. The existence part of Theorem \ref{th:main_theorem} (part (1)) then follows from the Banach fixed point theorem. The asymptotic formula for the resonance (part (2) of Theorem \ref{th:main_theorem}) using the formula for the fixed point of $\Psi$ as the limit $\lim_{n \rightarrow \infty} \Psi^n(z)$ for any $z \in \Omega$. The proof that $\Psi$ is a contraction requires precise estimates on solutions of \eqref{eq:Schro} which we obtain using ODE methods (specifically, Floquet theory). As a result, our proof does not generalize in a straightforward way to higher dimensions.

\begin{remark}
For $x \geq \rho$, the ODE \eqref{eq:eq} has periodic coefficients. Since by assumption the structure modeled by \eqref{eq:eq} has a bound state with energy $E$, Floquet theory implies the ODE must have real characteristic multipliers whose product is $1$. The $k > 0$ which appears in Theorem \ref{th:main_theorem} is simply the logarithm of the larger of these characteristic multipliers. For more details, see Section \ref{sec:bound_de_z_thet}.
\end{remark}

The asymptotic formula for the root \eqref{eq:asymp} is not very explicit. We can derive a more explicit formula as follows. When $\Phi$ is even, we define $v_z(x) \in C^\infty([0,M])$ to be the solution of
\begin{equation} \label{eq:eq_v}
    ( D_x^2 + V(x) - z ) v_z = 0, \quad v_z(0) = 0, v_z'(0) = 1
\end{equation}
for arbitrary $z \in \field{C}$. When $\Phi$ is odd, we make the same definitions except the initial data for \eqref{eq:eq_v} should be changed to $v_z(0) = -1, v_z'(0) = 0$. In either case $u_z(x)$ \eqref{eq:eq} and $v_z(x)$ form a fundamental solution set of the ODE (with Wronskian $u_z v_z' - u_z' v_z = 1$) appearing in \eqref{eq:eq_v}. We then have the following formula for the first-order correction appearing in \eqref{eq:asymp}:

\begin{corollary} \label{cor:imag_asymp}
The first correction term appearing in \eqref{eq:asymp} satisfies 
\begin{equation} \label{eq:correction}
    \frac{ \Theta(E) }{ \de_z \Theta(E) } = \frac{ i \sqrt{E} - \left[ u_E'(M) v_E'(M) + E u_E(M) v_E(M) \right] }{ \left[ ( v_E'(M) )^2 + E ( v_E(M) )^2 \right] \inty{0}{M}{ u_E^2 (y) }{y} } + O(M e^{- 4 k M}),
\end{equation}
\end{corollary}
\begin{proof}
The proof is a straightforward manipulation. See Section \ref{sec:Theta} for details.
\end{proof}

The consequences of Theorem \ref{th:main_theorem} and Corollary \ref{cor:imag_asymp} can now be summarized as follows:
\begin{corollary} \label{cor:summary}
Let $V, \Phi$, $E > 0$, $\rho > 0$, $u_z$, and $v_z$ be as in Theorem \ref{th:main_theorem} and Corollary \ref{cor:imag_asymp}, and let $H_{\rm trunc}$ be as in \eqref{eq:H_trunc}. Then there exists $k > 0$ and an $M_0 > \rho \geq 0$ such that for all $M \geq M_0$,
$H_{\rm trunc}$ has a resonance $z^*$ such that 
\begin{equation} \label{eq:res_pol}
\begin{split}
    \text{\emph{Re} } z^* &= E + \frac{ u_E'(M) v_E'(M) + E u_E(M) v_E(M) }{ \left[ ( v_E'(M) )^2 + E ( v_E(M) )^2 \right] \inty{0}{M}{ u_E^2 (y) }{y} } + O(M e^{- 4 k M})     \\
    \text{\emph{Im} } z^* &= \frac{ - \sqrt{E} }{ \left[ ( v_E'(M) )^2 + E ( v_E(M) )^2 \right] \inty{0}{M}{ u_E^2 (y) }{y} } + O(M e^{- 4 k M}).     \\
\end{split}
\end{equation}
The associated resonant state $\Phi^*(x)$ equals 
\begin{equation}
    \Phi^*(x) = \begin{cases} u_{z^*}(x), & 0 \leq x \leq M \\ u_{z^*}(M) e^{i \sqrt{z^*} (x - M)}, & \phantom{|}x\phantom{|} > M \end{cases}
\end{equation}
for $x \geq 0$. When $\Phi(x)$ is even (resp. odd), $\Phi^*(x)$ is even (resp. odd).
\end{corollary}

\begin{remark}
As $M \to \infty$ the resonance pole \eqref{eq:res_pol} will converge to the defect state eigenvalue, which explains why one observes that $\Omega_M$ decreases as $M \to \infty$.  In general, we suspect also that the transmission and reflection coefficients for plane wave scattering in the truncated problem converges to demonstrate the bands and the gaps of the model continuous spectrum in the infinite system (meaning the transmission coefficient should converge to $0$ where there are gaps in the fully periodic model spectrum).  Near energy $E$, the $M \to \infty$ model will have a gap, but, due to the existence of this resonance, convergence in that region may be non-uniform as observed in the related results \cite{DVW:15,duchene2015oscillatory}.  Questions of solutions for the transmission/reflection coefficients are related to solving for the so-called ``distorted Fourier Transform'' for all positive energies $E = k^2$ in a truncated lattice model with a defect state.  We leave computing these scattering state solutions to future works. 
\end{remark}

\section{Fixed-point argument} \label{sec:fixed}

In this section we will show how existence of a resonance nearby in the complex plane to the bound state eigenvalue $E$ can be deduced from a fixed point argument. Recall the definitions \eqref{eq:eq}-\eqref{eq:theta} of $u_z(x)$ and $\Theta(z)$.  Assuming that $\de_z \Theta(E) \neq 0$ (we will verify this below, see \eqref{eq:bound_de_z_thet}), $\Theta(z^*) = 0$ is equivalent to
\begin{equation} \label{eq:psi_z}
    \Psi(z^*) = z^*, \text{ where } \Psi(z) := z - \frac{ \Theta(z) }{ \de_z \Theta(E) }.
\end{equation}
In other words, $z^*$ is a resonance if and only if $z^*$ is a fixed point of the map $\Psi : \field{C} \rightarrow \field{C}$ defined by \eqref{eq:psi_z}. 

We will prove the existence of a fixed point nearby to $E$ by showing that $\Psi$ is a contraction in an appropriate ball centered at $E$ for $M$ sufficiently large. We start by defining a ball centered at $E$ with radius $f(M)$:
\begin{equation} \label{eq:omega_M}
    \Omega_M := \left\{ | z - E | \leq f(M) \right\}.
\end{equation}
Here we assume $f(M)$ is a function such that $\lim_{M \rightarrow \infty} f(M) = 0$ but don't define its precise form yet. By the mean value theorem, $\Psi$ is a contraction as long as $\Psi : \Omega_M \rightarrow \Omega_M$ and
\begin{equation} \label{eq:contract}
    \left| \de_z \Psi(z) \right| \leq \frac{1}{2} \text{ for all } z \in \Omega_M.
\end{equation}
Using the definition of $\Psi$ \eqref{eq:psi_z}, the condition for $\Psi$ to be a contraction \eqref{eq:contract} will be satisfied if
\begin{equation} \label{eq:bd}
    | \de_z \Theta(z) - \de_z \Theta(E) | \leq \frac{1}{2} | \de_z \Theta(E) | \text{ for all } z \in \Omega_M.
\end{equation}
Invoking the mean value theorem again, we have that
\begin{equation} \label{eq:MVT}
    | \de_z \Theta(z) - \de_z \Theta(E) | \leq f(M)  \sup_{z \in \Omega_M} \left| \de_z^2 \Theta(z) \right|  \text{ for all } z \in \Omega_M.
\end{equation}
Combining \eqref{eq:MVT} and \eqref{eq:bd} establishes the following criterion for $\Psi$ \eqref{eq:psi_z} to be a contraction in the ball $\Omega_M$ \eqref{eq:omega_M}.
\begin{lemma} \label{lem:contract}
Let $u_z$ be as in \eqref{eq:eq}, $\Theta$ be as in \eqref{eq:theta}, $\Psi$ be as in \eqref{eq:psi_z}, and $\Omega_M$ be as in \eqref{eq:omega_M}. Then $\Psi : \Omega_M \rightarrow \Omega_M$ is a contraction if
\begin{equation} \label{eq:for_fixed}
    f(M)  \sup_{z \in \Omega_M} \left| \de_z^2 \Theta(z) \right| \leq \frac{1}{2} | \de_z \Theta(E) |.
\end{equation}
\end{lemma}
We will prove the criterion \eqref{eq:for_fixed} holds for sufficiently large $M$ and $f(M) = \frac{1}{M^2} e^{- k M}$ by proving the following lemmas
\begin{lemma} \label{lem:bound_de_z_thet}
There exist positive real constants $C$ and $k$ and $M_0 > \rho > 0$ such that for all $M \geq M_0$, 
\begin{equation} \label{eq:bound_thet}
    | \Theta(E) | \leq C e^{- k M},
\end{equation}
and
\begin{equation} \label{eq:bound_de_z_thet}
    | \de_z \Theta(E) | \geq C e^{k M}. 
\end{equation}
\end{lemma}
\begin{proof}
Given in Section \ref{sec:bound_de_z_thet}.
\end{proof}
\begin{lemma} \label{lem:bound_de_z_2_thet}
Let $k$ be as in Lemma \ref{lem:bound_de_z_thet} and $f(M)$ in the definition of $\Omega_M$ be $\frac{1}{M^2} e^{- k M}$. Then there exist positive constants $C$ and $M_0 > \rho \geq 0$ such that for all $M \geq M_0$,
\begin{equation}
    \sup_{z \in \Omega_M} \left| \de_z^2 \Theta(z) \right| \leq C e^{k M}.
\end{equation}
\end{lemma}
\begin{proof}
Given in Section \ref{sec:bound_de_z_2_thet}.
\end{proof}
We can now prove the existence part of Theorem \ref{th:main_theorem}.
\begin{proof}[Proof of part (1) of Theorem \ref{th:main_theorem}]
Assuming their hypotheses, Lemmas \ref{lem:bound_de_z_thet} and \ref{lem:bound_de_z_2_thet} clearly imply the condition \eqref{eq:for_fixed}. To see that $\Psi : \Omega_M \rightarrow \Omega_M$, note that using \eqref{eq:contract} we have
\begin{equation}
    \left| \Psi(z) - E \right| = \left| \Psi(z) - \Psi(E) + \Psi(E) - E \right| \leq \frac{1}{2} |z - E| + \left| \frac{\Theta(E)}{\de_z \Theta(E)} \right|.
\end{equation}
If $z \in \Omega_M$, then by definition and by Lemma \ref{lem:bound_de_z_thet} we have that
\begin{equation}
    \left| \Psi(E) - E \right| \leq \frac{1}{2 M^2} e^{- k M} + C e^{- 2 k M},
\end{equation}
and hence for sufficiently large $M$ we have $\Psi(z) \in \Omega_M$. Hence by Lemma \ref{lem:contract} we are done.
\end{proof}
The proofs of Lemmas \ref{lem:bound_de_z_thet} and \ref{lem:bound_de_z_2_thet} involve careful estimation of solutions of \eqref{eq:eq} and are the subjects of Sections \ref{sec:bound_de_z_thet} and \ref{sec:bound_de_z_2_thet}. 

\section{Asymptotics of the resonance} \label{sec:asymp}

We now show how to prove the asymptotic formula part of Theorem \ref{th:main_theorem} assuming Lemmas \ref{lem:bound_de_z_thet} and \ref{lem:bound_de_z_2_thet}. The idea is to use the formula for the resonance in terms of iterates of the contraction $\Psi$:
\begin{equation} \label{eq:iterate}
    z^* = \lim_{n \rightarrow \infty} \Psi^n(E),
\end{equation}
although just one iterate will be enough to capture the leading-order asymptotics. 

\begin{proof}[Proof of part (2) of Theorem \ref{th:main_theorem}]
Let $z^{(n)} := \Psi^n(E)$. Since $\Psi$ is a contraction we have that
\begin{equation}
    | z^{(n+1)} - z^{(n)} | \leq \frac{1}{2} | z^{(n)} - z^{(n-1)} |.
\end{equation}
By a telescoping argument it now follows that 
\begin{equation} \label{eq:tele}
    | z^* - z^{(1)} | = \left| \sum_{n = 1}^\infty z^{(n+1)} - z^{(n)} \right| \leq 2 | z^{(2)} - z^{(1)} |.
\end{equation}
By definition,
\begin{equation} \label{eq:z1}
    z^{(1)} = E - \frac{ \Theta(E) }{ \de_z \Theta(E) },
\end{equation}
and 
\begin{equation} \label{eq:z2}
    z^{(2)} = z^{(1)} - \frac{ \Theta( z^{(1)} ) }{ \de_z \Theta(E) }.
\end{equation}
Substituting \eqref{eq:z1} and \eqref{eq:z2} into \eqref{eq:tele} implies
\begin{equation} \label{eq:almost}
    \left| z^* -  \left( E - \frac{ \Theta(E) }{ \de_z \Theta(E) } \right) \right| \leq 2 \left| \frac{ \Theta\left( E - \frac{ \Theta(E) }{ \de_z \Theta(E) } \right) }{ \de_z \Theta(E) } \right|.
\end{equation}
The asymptotic formula \eqref{eq:asymp} will follow immediately if we can prove that the right-hand side of \eqref{eq:almost} is $O(e^{- 4 k M})$. To this end, note that second-order Taylor expansion implies that
\begin{equation} \label{eq:taylor}
\begin{split}
    \left| \Theta\left( E - \frac{ \Theta(E) }{ \de_z \Theta(E) } \right) \right| &= \left| \Theta\left( E - \frac{ \Theta(E) }{ \de_z \Theta(E) } \right) - \Theta(E) + \frac{ \Theta(E) }{ \de_z \Theta(E) } \de_z \Theta(E) \right|    \\
    &\leq \frac{1}{2} \left| \frac{ \Theta(E) }{ \de_z \Theta(E) } \right|^2 \sup_{z \in \Omega_M} |\de_z^2 \Theta(z)|.
\end{split}
\end{equation}
Substituting \eqref{eq:taylor} into \eqref{eq:almost} now gives
\begin{equation} \label{eq:almost_2}
    \left| z^* -  \left( E - \frac{ \Theta(E) }{ \de_z \Theta(E) } \right) \right| \leq \frac{1}{ |\de_z \Theta(E)| } \left| \frac{ \Theta(E) }{ \de_z \Theta(E) } \right|^2 \sup_{z \in \Omega_M} |\de_z^2 \Theta(z)|.
\end{equation}
Under the hypotheses of Lemmas \ref{lem:bound_de_z_thet} and \ref{lem:bound_de_z_2_thet} we have that $| \de_z \Theta(E) | \geq C e^{k M}$, $\sup_{z \in \Omega_M} | \de_z^2 \Theta(z) | \leq C e^{k M}$, and $| \Theta(E) | \leq C e^{- k M}$, for positive constants $C > 0$. The asymptotic formula \eqref{eq:asymp} now follows immediately from substituting these estimates into \eqref{eq:almost_2}.
\end{proof}

\section{Proof of Lemma \ref{lem:bound_de_z_thet}: Bounding $\Theta(E)$ from above and $\de_z \Theta(E)$ from below} \label{sec:bound_de_z_thet}

In this section we will prove Lemma \ref{lem:bound_de_z_thet}. WLOG we assume that the bound state $\Phi(x)$ is even, since the case where the bound state is odd is similar. Recall that we define 
\begin{equation} \label{eq:thet}
    \Theta(z) = u_z'(M) - i \sqrt{z} u_z(M), 
\end{equation}
where $u_z(x) \in C^\infty([0,M])$ is defined for any $z \in \field{C}$ by
\begin{equation} \label{eq:eq_2}
    ( D_x^2 + V(x) - z ) u_z = 0, \quad u_z(0) = 1, u_z'(0) = 0.
\end{equation}

Basic Floquet theory (see e.g. \cite{MagnusWinkler}) implies that the ODE appearing in \eqref{eq:eq_2} has a fundamental solution set $u_{1,z}(x)$ and $u_{2,z}(x)$ which satisfy
\begin{equation} \label{eq:quasi}
    u_{1,z}(x + 1) = \lambda(z) u_{1,z}(x) \quad u_{2,z}(x + 1) = \lambda(z)^{-1} u_{2,z}(x) \text{ for all $x \geq \rho$} 
\end{equation}
for some $\lambda(z) \in \field{C}$ depending on $z$. For $H$ to have a bound state with energy $E$, it must be that \eqref{eq:eq_2} has a decaying solution when $z = E$ and hence $|\lambda(E)| \neq 1$. Without loss of generality, we will assume $0 < \lambda(E) < 1$ so that $u_{1,E}(x)$ decays, and $u_{2,E}(x)$ grows, as $x \rightarrow \infty$. It is convenient to introduce a new constant $k > 0$ such that
\begin{equation}
    \lambda(E) = e^{- k},
\end{equation}
so that when $z = E$, \eqref{eq:quasi} becomes
\begin{equation} \label{eq:quasi_2}
    u_{1,E}(x + 1) = e^{- k} u_{1,E}(x) \quad u_{2,E}(x + 1) = e^k u_{2,E}(x) \text{ for all $x \geq \rho$}.
\end{equation}

Since the bound state must decay as $x \rightarrow \infty$, and (by the assumption that the bound state is even) satisfy the initial conditions \eqref{eq:eq_2}, we have that (perhaps after multiplying $u_{1,E}(x)$ by a constant)
\begin{equation}
    u_E(x) = u_{1,E}(x)
\end{equation}
for all $x \in [0,M]$. Furthermore, letting $v_z(x)$ denote the solution of the following initial value problem
\begin{equation}
    ( D_x^2 + V(x) - z ) v_z = 0, \quad v_z(0) = 0, v_z'(0) = 1,
\end{equation}
we must have that (perhaps after multiplying $u_{2,E}(x)$ by a constant)
\begin{equation}
    v_E(x) = u_{2,E}(x).
\end{equation}

We can now state two equivalent identities which are important in the coming proofs. For $\rho \leq x \leq M$, let $[x]$ be defined by the conditions
\begin{equation} \label{eq:x_sq}
    [x] \in [\rho,\rho+1) \text{ and } \exists m \in \field{N} \text{ such that } [x] + m = x.
\end{equation}
Then \eqref{eq:quasi_2} implies immediately that
\begin{equation} \label{eq:sols_1}
    u_E(x) = e^{- k (x - [x])} u_E([x]), \text{ and } v_E(x) = e^{k (x - [x])} v_E([x]).
\end{equation}
Equivalently, there exist 1-periodic functions $p(x), q(x) \in C^\infty([\rho,M])$ such that
\begin{equation} \label{eq:sols_2}
    u_{E}(x) = e^{- k (x - \rho)} p(x), \text{ and } v_{E}(x) = e^{k (x - \rho)} q(x) \text{ for all $x \geq \rho$}.
\end{equation}

The following lemma is now clear 
\begin{lemma} \label{lem:bound_uz}
There exist positive constants $C_1, C_2, k$ such that for all $M \geq \rho$
\begin{equation}
    \left| u_E(M) \right| \leq C_1 e^{- k M}, \quad \left| u_E'(M) \right| \leq C_2 e^{- k M}.
\end{equation}
\end{lemma}
We can now give the proof of assertion \eqref{eq:bound_thet} of Lemma \ref{lem:bound_de_z_thet}.
\begin{proof}[Proof of assertion \eqref{eq:bound_thet} of Lemma \ref{lem:bound_de_z_thet}]
The assertion follows immediately from directly bounding \eqref{eq:thet} using the triangle inequality and Lemma \ref{lem:bound_uz}.
\end{proof}

Towards a proof of assertion \eqref{eq:bound_de_z_thet} of Lemma \ref{lem:bound_de_z_thet}, we differentiate $\Theta(z)$ with respect to $z$:
\begin{equation}
    \de_z \Theta(z) := \de_z u_z'(M) - i \sqrt{z} \de_z u_z(M) - i \frac{1}{2 \sqrt{z}} u_z(M).
\end{equation}
Differentiating \eqref{eq:eq_2} with respect to $z$ we see that $\de_z u_z(x)$ satisfies
\begin{equation} \label{eq:nonhom_eq}
    ( D_x^2 + V(x) - z ) \de_z u_z = u_z, \quad \de_z u_z (0) = 0, \de_z u_z'(0) = 0
\end{equation}
for all $z \in \field{C}$. When $z = E$, we can construct the solution of \eqref{eq:nonhom_eq} in terms of $u_E(x)$ and $v_E(x)$ using variation of parameters. This will allow us to prove the following Lemma which is the key step in bounding $\de_z \Theta(E)$ below. 
\begin{lemma} \label{lem:for_bd_bel}
Let $k > 0$ be as in Lemma \ref{lem:bound_uz}. Then there exist constants $C > 0$ and $M_0 \geq \rho$ such that for all $M \geq M_0$,
\begin{equation} \label{eq:bd_below}
    \left( \left( \left. \de_z u_z'(M) \right|_{z = E} \right)^2 + E \left( \left. \de_z u_z(M) \right|_{z = E} \right)^2 \right)^{1/2} \geq C e^{k M}.
\end{equation}
\end{lemma}
\begin{proof}
Given in Section \ref{sec:ests_2}.
\end{proof}
We can now prove the second assertion of Lemma \ref{lem:bound_de_z_thet}.
\begin{proof}[Proof of assertion \eqref{eq:bound_de_z_thet} of Lemma \ref{lem:bound_de_z_thet}]
Since every term of \eqref{eq:nonhom_eq} is real when $z = E$, $\left. \de_z u_z(M) \right|_{z = E}$ and $\left. \de_z u_z'(M) \right|_{z = E}$ are real. It then follows that 
\begin{equation} 
    \left| \de_z u_z'(M) - i \sqrt{z} \de_z u_z(M) \right|_{z = E} = \left( \left( \left. \de_z u_z'(M) \right|_{z = E} \right)^2 + E \left( \left. \de_z u_z(M) \right|_{z = E} \right)^2 \right)^{1/2}.
\end{equation}
Now using Lemmas \ref{lem:bound_uz} and \ref{lem:for_bd_bel} and the reverse triangle inequality we can compute
\begin{equation}
\begin{split}
    | \de_z \Theta(E) | &= \left| \de_z u_z'(M) - i \sqrt{z} \de_z u_z(M) - i \frac{1}{2 \sqrt{z}} u_z(M) \right|_{z = E}    \\
    &\geq \left| \de_z u_z'(M) - i \sqrt{z} \de_z u_z(M) \right|_{z = E} + O(e^{-k M})  \\
    &\geq C e^{k M}
\end{split}
\end{equation}
for sufficiently large $M$.
\end{proof}

We now give the proof of Lemma \ref{lem:for_bd_bel}.

\subsection{Proof of Lemma \ref{lem:for_bd_bel}} \label{sec:ests_2}

Using variation of parameters and noting that the Wronskian of $u_z$ and $v_z$ is 1, we have for any $z \in \field{C}$ that 
\begin{equation}
    \de_z u_z(x) = \left( \inty{0}{x}{ v_z(y) u_z(y) }{y} \right) u_z(x) - \left( \inty{0}{x}{ u_z^2(y) }{y} \right) v_z(x).
\end{equation}
Differentiating with respect to $x$ gives
\begin{equation}
\begin{split}
    \de_z u_z'(x) &= \left( v_z(x) u_z(x) \right) u_z(x) + \left( \inty{0}{x}{ v_z(y) u_z(y) }{y} \right) u_z'(x) \\
    & \qquad - \left( u_z^2(x) \right) v_z(x) - \left( \inty{0}{x}{ u_z^2(y) }{y} \right) v_z'(x)  \\
    &= \left( \inty{0}{x}{ v_z(y) u_z(y) }{y} \right) u_z'(x) - \left( \inty{0}{x}{ u_z^2(y) }{y} \right) v_z'(x).
\end{split}
\end{equation}
Evaluating these expressions at $x = M$ and $z = E$ we have
\begin{equation}
\begin{split}
    \left. \de_z u_z(M) \right|_{z = E} &= \left( \inty{0}{M}{ v_E(y) u_E(y) }{y} \right) u_E(M) - \left( \inty{0}{M}{ u_E^2(y) }{y} \right) v_E(M)  \\
    \left. \de_z u_z'(M) \right|_{z = E} &= \left( \inty{0}{M}{ v_E(y) u_E(y) }{y} \right) u_E'(M) - \left( \inty{0}{M}{ u_E^2(y) }{y} \right) v_E'(M).
\end{split}
\end{equation}
Using \eqref{eq:sols_2} we have that $v_E(x) u_E(x)$ is 1-periodic for $x \geq \rho$. Combining this with Lemma \ref{lem:bound_uz} we have that
\begin{equation} \label{eq:dis}
\begin{split}
    \left. \de_z u_z(M) \right|_{z = E} &= - \left( \inty{0}{M}{ u_E^2(y) }{y} \right) v_E(M) + O( M e^{- k M} )  \\
    \left. \de_z u_z'(M) \right|_{z = E} &= - \left( \inty{0}{M}{ u_E^2(y) }{y} \right) v_E'(M) + O( M e^{- k M} ).
\end{split}
\end{equation}
We can now start to prove \eqref{eq:bd_below}. Using \eqref{eq:dis} we have that 
\begin{equation} \label{eq:to_bb}
\begin{split}
    &\left( \left( \left. \de_z u_z'(M) \right|_{z = E} \right)^2 + E \left( \left. \de_z u_z(M) \right|_{z = E} \right)^2 \right)^{1/2} \\
    &= \left( \left[ \left( \inty{0}{M}{ u_E^2(y) }{y} \right) v_E(M) + O( M e^{- k M} ) \right]^2  \right.  \\
    & \hspace{.5cm} \left. + \left[ \left( \inty{0}{M}{ u_E^2(y) }{y} \right) v_E'(M) + O( M e^{- k M} ) \right]^2 \right)^{1/2}.
\end{split}
\end{equation}
Using the reverse triangle inequality we have that 
\begin{align} 
&  \text{Equation \eqref{eq:to_bb}} \geq \notag \\
&  \left| \left( \left[ \left( \inty{0}{M}{ u_E^2(y) }{y} \right) v_E(M) \right]^2 + \left[ \left( \inty{0}{M}{ u_E^2(y) }{y} \right) v_E'(M) \right]^2 \right)^{1/2} - O(M e^{- k M}) \right|.  \label{eq:nex}
\end{align}
We are done if we can show the first term inside the absolute value signs can be bounded below by $C e^{k M}$ where $C > 0$ is a positive constant independent of $M$. 

We start by noting that $\inty{0}{M}{ u_E^2(y) }{y}$ is clearly a non-zero and increasing function of $M$, and can hence can be bounded below by a constant independent of $M$, for example
\begin{equation} \label{eq:simpp}
    \inty{0}{M}{ u_E^2(y) }{y} \geq \inty{0}{\rho}{ u_E^2(y) }{y} > 0.
\end{equation}
It follows that
\begin{equation} \label{eq:cont}
\begin{split}
    &\left( \left[ \left( \inty{0}{M}{ u_E^2(y) }{y} \right) v_E(M) \right]^2 + \left[ \left( \inty{0}{M}{ u_E^2(y) }{y} \right) v_E'(M) \right]^2 \right)^{1/2}     \\
    &\geq \left( \inty{0}{\rho}{ u_E^2(y) }{y} \right) \left( \left[ v_E(M) \right]^2 + \left[ v_E'(M) \right]^2 \right)^{1/2}.
\end{split}
\end{equation}
Using the identity \eqref{eq:sols_1} we have 
\begin{equation} \label{eq:exp_behavior}
\begin{split}
    &v_E(M) = e^{k (M - [M])} v_E([M]), \text{ and } v_E'(M) = e^{k (M - [M])} v_E'([M]) \\
    \implies& \; |v_E(M)| \geq e^{k M} e^{- k (\rho + 1)} \inf_{x \in [\rho,\rho+1)} |v_E(x)| \\ 
    &\text{ and } |v_E'(M)| \geq e^{k M} e^{- k (\rho + 1)} \inf_{x \in [\rho,\rho+1)} |v_E'(x)|,
\end{split}
\end{equation}
and hence \eqref{eq:cont} can be bounded below by
\begin{equation} \label{eq:bbd}
\begin{split}
    &\geq \left( \inty{0}{\rho}{ u_E^2(y) }{y} \right) e^{k M} e^{- k (\rho + 1)} \inf_{x \in [\rho,\rho+1)} \left( \left[ v_E(x) \right]^2 + \left[ v_E'(x) \right]^2 \right)^{1/2}.
\end{split}
\end{equation}
We can bound $\left( [v_E(x)]^2 + [v_E'(x)]^2 \right)^{1/2}$ below uniformly in $x \in [\rho,\rho+1)$ using the Wronskian. For any $x \in [\rho,\rho+1)$, we have by the Cauchy-Schwarz inequality
\begin{align}
    & 1 = \left| u_E'(x) v_E(x) - u_E(x) v_E'(x) \right| \notag  \\
   & \hspace{1cm} \leq \left( [u_E(x)]^2 + [u_E'(x)]^2 \right)^{1/2} \left( [v_E(x)]^2 + [v_E'(x)]^2 \right)^{1/2}, \notag
\end{align}
and hence
\begin{equation} \label{eq:v_bd_bel}
    \left( [v_E(x)]^2 + [v_E'(x)]^2 \right)^{1/2} \geq \frac{1}{ \sup_{x \in [\rho,\rho+1)} \left( [u_E(x)]^2 + [u_E'(x)]^2 \right)^{1/2} }.
\end{equation}
Combining \eqref{eq:v_bd_bel} with \eqref{eq:bbd} and substituting into \eqref{eq:nex} we are done.

\section{Proof of Lemma \ref{lem:bound_de_z_2_thet}: Bounding $\de_z^2 \Theta(z)$ uniformly for $z \in \Omega_M$}
\label{sec:bound_de_z_2_thet}

We now seek to bound $\de_z^2 \Theta(z)$ uniformly in $z$ in a ball centered at $E$. Differentiating $\Theta(z)$ gives
\begin{equation} \label{eq:to_bd}
    \de_z^2 \Theta(z) = \de_z^2 u_z'(M) + \frac{1}{4} i z^{-3/2} u_z(M) - i z^{-1/2} \de_z u_z(M) - i z^{1/2} \de_z^2 u_z(M),
\end{equation}
and hence we are done if we can bound $\de_z^2 u_z'(M)$, $u_z(M)$, $\de_z u_z(M)$, and $\de_z^2 u_z(M)$ uniformly in $z$ for $z$ nearby to $E$. 

By differentiating the initial value problem \eqref{eq:eq} we have that $\de_z^2 u_z$ satisfies
\begin{equation} \label{eq:dz2_u_zee}
    ( D_x^2 + V(x) - z ) \de_z^2 u_z = 2 \de_z u_z \quad \de_z^2 u(0) = \de_z^2 u'(0) = 0,
\end{equation}
while $\de_z u_z$ and $u_z$ satisfy 
\begin{equation} \label{eq:dz_u_zee}
    ( D_x^2 + V(x) - z ) \de_z u_z = u_z, \quad \de_z u_z (0) = 0, \de_z u_z'(0) = 0,
\end{equation}
\begin{equation} \label{eq:u_zee}
    ( D_x^2 + V(x) - z ) u_z = 0, \quad u_z (0) = 1, u_z'(0) = 0,
\end{equation}
respectively. The key step in the proof of Lemma \ref{lem:bound_de_z_thet} is the following:
\begin{lemma} \label{lem:bds}
Let $x > 0$ be arbitrary, and let $z \in \field{C}$ be such that
\begin{equation}
    |z - E| \leq \frac{1}{2 x C^* (e^{k x} + 1)}.
\end{equation}
Then there exists a positive constant $C^* > 0$ such that the following estimates hold
\begin{equation} \label{eq:bound_on_u}
    \max \left\{ | {u}_z(x) |, | u_z'(x) | \right\} \leq 2 C^* ( e^{k x} + 1).
\end{equation}
\begin{equation} \label{eq:bound_dz_u}
    \max \left\{ | \de_z {u}_z(x) |, | \de_z u_z'(x) | \right\} \leq 4 C^* ( e^{k x} + 1)
\end{equation}
\begin{equation} \label{eq:bound_dz2_u}
    \max \left\{ | \de_z^2 {u}_z(x) |, | \de_z^2 u_z'(x) | \right\} \leq 8 C^* ( e^{k x} + 1).
\end{equation}
\end{lemma}
\begin{proof}
The proofs of \eqref{eq:bound_on_u} and \eqref{eq:bound_dz_u} are given in Sections \ref{sec:u_zee} and \ref{sec:dz_u_zee}. Since the proof of \eqref{eq:bound_dz2_u} is so similar to that of \eqref{eq:bound_dz_u}, we omit it.
\end{proof}
We can now give the proof of Lemma \ref{lem:bound_de_z_2_thet}.
\begin{proof}[Proof of Lemma \ref{lem:bound_de_z_2_thet}]
If $f(M) = \frac{1}{M^2} e^{- k M}$, then there exists a $M_0 > 0$ such that for all $M \geq M_0$, $z \in \Omega_M$ implies that
\begin{equation}
    |z - E| \leq \frac{1}{2 M C^* (e^{k M} + 1)}.
\end{equation}
Applying Lemma \ref{lem:bds} to \eqref{eq:to_bd} and using the triangle inequality yields
\begin{align}
     | \de_z \Theta(z) | & \leq 8 C^* (e^{k M} + 1) + \frac{1}{4} |z|^{-3/2} 2 C^* (e^{k M} + 1)  \notag \\
   & \hspace{.25cm} + |z|^{-1/2} 4 C^* (e^{k M} + 1) + |z|^{1/2} 8 C^* (e^{k M} + 1). \notag
\end{align}
To bound the terms involving $z$ we note that
\begin{equation}
    |z|^{1/2} = |z - E + E|^{1/2} \leq |E|^{1/2} + C |z - E|
\end{equation}
where $C = \sup_{z \in \Omega_M} \fdf{z} |z|^{1/2}$. Since we assume $E > 0$, we have that (perhaps after taking $M_0$ larger)
\begin{equation}
    | \de_z \Theta(z) | \leq C e^{k M}
\end{equation}
for some constant $C > 0$.
\end{proof}

\subsection{Proof of assertion (\ref{eq:bound_on_u}) of Lemma \ref{lem:bds}: Bounding solutions of (\ref{eq:u_zee}) uniformly for $z \in \Omega_M$}
\label{sec:u_zee}

The idea is to use Picard iteration to solve \eqref{eq:u_zee} perturbatively about $z = E$. We start by writing \eqref{eq:u_zee} as
\begin{equation} 
    ( D_x^2 + V(x) - E ) u_z = (z - E) u_z, \quad u_z (0) = 1, u_z'(0) = 0,
\end{equation}
and then as a first-order system for $\vec{u}_z := (u_z(x),u_z'(x))^\top$:
\begin{equation} \label{eq:first_ord}
   \vec{u}_z ' = H(x) \vec{u}_z + \tilde{H} \vec{u}_z, \quad \vec{u}_z(0) = (1,0)^\top,
\end{equation}
where
\begin{equation} \label{eq:H_H_tilde}
    H(x) := \begin{pmatrix} 0 & 1 \\ V(x) - E & 0 \end{pmatrix}, \quad \tilde{H} := \begin{pmatrix} 0 & 0 \\ E - z & 0 \end{pmatrix}.
\end{equation}
Note that $H(x + 1) = H(x)$ for $x \geq \rho$. 

When $z = E$ so that $\tilde{H} = 0$, \eqref{eq:first_ord} has the solution
\begin{equation}
    \vec{u}_z(x) = U_E(x) \vec{u}_z(0) 
\end{equation}
where $U_E(x)$ is the solution operator
\begin{equation}
    U_E(x) = \begin{pmatrix} u_E(x) & v_E(x) \\ u_E'(x) & v_E'(x) \end{pmatrix}.
\end{equation}
Recalling \eqref{eq:sols_2}, we have that 
\begin{equation} 
    u_E(x) = e^{- k (x - \rho)} p(x), \quad u_E'(x) = e^{- k (x - \rho)} \left( - k p(x) + p'(x) \right),
\end{equation}
where $p(x) \in C^\infty(\field{R})$ is 1-periodic, and $k > 0$, for all $x \geq \rho$. Similarly, we have that 
\begin{equation}
    v_E(x) = e^{k (x - \rho)} q(x), \quad v_E'(x) = e^{k (x - \rho)} \left( k q(x) + q'(x) \right),
\end{equation}
where $q(x)$ is 1-periodic, and $k > 0$, for all $x \geq \rho$. It is now trivial to prove pointwise bounds on each of the entries of $U_E(x)$, for example 
\begin{equation}
    |v_E(x)| \leq \sup_{0 \leq x \leq \rho} |v_E(x)| + e^{k (x - \rho)} \sup_{\rho \leq x \leq \rho + 1} |q(x)|.
\end{equation}
From these pointwise bounds it is clear that there exists a constant $C^* > 0$ such that 
\begin{equation}
\begin{split}
    \left\| U_E(x) \vec{f} \right\|_\infty &= \sup\left\{ \left| u_E(x) f_1 + v_E(x) f_2 \right| , \left| u_E'(x) f_1 + v_E'(x) f_2 \right| \right\}    \\
    &\leq 2 \sup \left\{ |u_E(x)|, |v_E(x)|, |u_E'(x)|, |v_E'(x)| \right\} \| \vec{f} \|_{\infty}   \\
    &\leq C^* \left( e^{k x} + 1 \right) \| \vec{f} \|_\infty.
\end{split}
\end{equation}

We now seek to solve \eqref{eq:first_ord} for $|z - E|$ small by Picard iteration. Using Duhamel's formula, we can re-write \eqref{eq:first_ord} as the fixed point equation
\begin{equation}
    \vec{u}_z(x) = U_E(x) \vec{u}_z(0) + \inty{0}{x}{ U_E(x - y) \tilde{H} \vec{u}_z(y) }{y}.
\end{equation}
For any $X \geq 0$, we define an operator
\begin{equation}
    T_X : \vec{f}(x) \mapsto U(x) \vec{f}(0) + \inty{0}{x}{ U(x-y) \tilde{H} \vec{f}(y) }{y} 
\end{equation}
acting on the Banach space 
\begin{equation} \label{eq:banach}
    B_X := \left\{ \vec{f} \in C([0,X];\field{C}^2) : \vec{f}(0) = \vec{u}_z(0) \right\},
\end{equation}
equipped with the sup norm. To see when this operator is a contraction we consider 
\begin{equation}
   \| T_X \vec{f}(x) - T_X \vec{g}(x) \|_{\infty} = \left\| \inty{0}{x}{ U(x-y) \tilde{H} \left( \vec{f}(y) - \vec{g}(y) \right) }{y} \right\|_{\infty}.
\end{equation}

Since $\| \tilde{H} \|_{\infty}$ is clearly bounded by $|z - E|$, we have
\begin{equation}
\begin{split}
   \| T_X \vec{f}(x) - T_X \vec{g}(x) \|_{\infty} &= \left\| \inty{0}{x}{ U(x-y) \tilde{H} \left( \vec{f}(y) - \vec{g}(y) \right) }{y} \right\|_{\infty} \\
   &\leq X \sup_{0 \leq x \leq X} \| U(x) \|_{\infty} |z - E| \| \vec{f}(y) - \vec{g}(y) \|_\infty \\
   &\leq C^* X (e^{k X} + 1) |z - E| \| \vec{f}(x) - \vec{g}(x) \|_{\infty}.
\end{split}
\end{equation}
It follows that $T_X$ is a contraction as long as 
\begin{equation}
    |z - E| < \frac{1}{ C^* X (e^{k X} + 1) }.
\end{equation}
In this case we have the following formula for $\vec{u}_z(x)$:
\begin{equation} \label{eq:form}
    \vec{u}_z(x) = \lim_{n \rightarrow \infty} T^n \vec{u}_{z,0},
\end{equation}
where $\vec{u}_{z,0}$ denotes the function equal to the constant $\vec{u}_z(0)$ on the interval $[0,X]$. \eqref{eq:form} written out is
\begin{equation} \label{eq:form_2}
\begin{split}
    \vec{u}_z(x)  &= U_E(x) \vec{u}_z(0) + \inty{0}{x}{ U_E(x-y) \tilde{H} U_E(y) \vec{u}_z(0) }{y} \\
    & + \inty{0}{x}{ U_E(x-y) \tilde{H} \inty{0}{y}{ U_E(y-y_1) \tilde{H} U_E(y_1) \vec{u}_z(0) }{y_1} }{y} + ...
\end{split}
\end{equation}
The second term on the right-hand side can be bounded as
\begin{equation}
\begin{split}
    \left| \inty{0}{x}{ U_E(x-y) \tilde{H} U_E(y) \vec{u}_z(0) }{y} \right| &\leq x \sup_{0 \leq y \leq x} \left| U_E(x-y) \tilde{H} U_E(y) \vec{u}_z(0) \right|   \\
    &\leq x C^* ( e^{k x} + 1 ) |z - E| C^* ( e^{k x} + 1 )  \\
    &= x |z - E| \left( C^* (e^{k x} + 1) \right)^2.
\end{split}
\end{equation}
Bounding the other terms on the right-hand side similarly, we have
\begin{equation}
\begin{split}
    & | \vec{u}_z(x) | \leq C^*( e^{k x} + 1) + x |z - E| (C^* (e^{k x} + 1) )^2 + x^2 |z - E|^2 ( C^* (e^{k x} + 1) )^3 + ...  \\
    &\hspace{.35cm} = C^*( e^{k x} + 1) \left( 1 + x |z - E| (C^* (e^{k x} + 1) ) + x^2 |z - E|^2 ( C^* (e^{k x} + 1) )^2 + ... \right).
\end{split}
\end{equation}
Hence, whenever $x |z - E| C^* (e^{k x} + 1) < 1$, we have that
\begin{equation} 
    | \vec{u}_z(x) | \leq \frac{ C^* ( e^{k x} + 1) }{ 1 - x |z - E| C^* ( e^{k x} + 1 ) }.
\end{equation}
By taking 
\begin{equation}
    |z - E| \leq \frac{1}{2 x C^* (e^{k x} + 1) }
\end{equation}
we have the desired bound \eqref{eq:bound_on_u}
\begin{equation} 
    | \vec{u}_z(x) | \leq 2 C^* ( e^{k x} + 1 ).
\end{equation}

\subsection{Proof of assertion (\ref{eq:bound_dz_u}) of Lemma \ref{lem:bds}: Bounding solutions of (\ref{eq:dz_u_zee}) uniformly for $z \in \Omega_M$}
\label{sec:dz_u_zee}

We start by re-writing \eqref{eq:dz_u_zee} as
\begin{equation} \label{eq:to_bddd}
    ( D_x^2 + V(x) - E ) \de_z u_z = (z - E) \de_z u_z + u_z, \quad \de_z u_z (0) = 0, \de_z u_z'(0) = 0.
\end{equation}
We can write \eqref{eq:to_bddd} as a first-order system just as in \eqref{eq:first_ord}. Using Duhamel's formula and the fact that $\de_z \vec{u}_z(0) = 0$, we have that
\begin{equation}
    \de_z \vec{u}_z(x) = \inty{0}{x}{ U(x - y) \left[ \tilde{H} \de_z \vec{u}_z(y) + \vec{u}_z(y) \right] }{y},
\end{equation}
where $\tilde{H}$ is as in \eqref{eq:H_H_tilde}. The map
\begin{equation}
    T_X' : \vec{f}(x) \mapsto \inty{0}{x}{ U_E(x-y) \left[ \tilde{H} \vec{f}(y) + \vec{u}_z(y) \right] }{y}
\end{equation}
is clearly a contraction on the Banach space of functions on $0 \leq x \leq X$ with $\vec{f}(0) = 0$ equipped with the sup norm under the same conditions as before i.e. as long as $X C^* ( e^{k X} + 1 ) |z - E| < 1$. Starting the iteration with the function equal to $0$ over the whole interval, we have the analogous representation as \eqref{eq:form_2} for the solution of \eqref{eq:to_bddd}:
\begin{equation} \label{eq:analogous}
\begin{split}
    &\de_z \vec{u}_z(x) = \inty{0}{x}{ U_E(x-y) \vec{u}_z(y) }{y} \\
    & + \inty{0}{x}{ U_E(x-y) \tilde{H} \inty{0}{y}{ U_E(y-y_1) \vec{u}_z(y_1) }{y_1} }{y} + ...
\end{split}
\end{equation}
Replacing $\vec{u}_z(y)$ everywhere by its expansion \eqref{eq:form_2}, we have for the first term on the right-hand side in \eqref{eq:analogous}
\begin{equation}
\begin{split}
    &\inty{0}{x}{ U_E(x-y) \vec{u}_z(y) }{y} \\
    &= \inty{0}{x}{ U_E(x) \vec{u}_z(0) }{y} + \inty{0}{x}{ U_E(x-y) \inty{0}{y}{ U_E(y - y_1) \tilde{H} U_E(y_1) \vec{u}_z(0) }{y_1} }{y} + ...,
\end{split}
\end{equation}
and hence
\begin{equation}
\begin{split}
    \inty{0}{x}{ U_E(x - y) \vec{u}_z(y) }{y} &\leq x C^* (e^{k x} + 1) + x^2 \left( C^* (e^{k x} + 1) \right)^2 |z - E| + ...   \\
    &\leq 2 x C^* (e^{k x} + 1)
\end{split}
\end{equation}
whenever $ x C^* (e^{k x} + 1) |z - E| \leq \frac{1}{2}$. As for the second term in \eqref{eq:analogous} we have
\begin{equation}
\begin{split}
    &\inty{0}{x}{ U_E(x-y) \tilde{H} \inty{0}{y}{ U_E(y-y_1) \vec{u}_z(y_1) }{y_1} }{y} \\
    &= \inty{0}{x}{ U_E(x-y) \tilde{H} \inty{0}{y}{ U_E(y) \vec{u}_z(0) }{y_1} }{y} \\
    & \hspace{.125cm} + \inty{0}{x}{ U_E(x-y) \tilde{H} \inty{0}{y}{ U_E(y - y_1) \inty{0}{y_1}{ U_E(y_1 - y_2) \tilde{H} U_E(y_2) \vec{u}_z(0) }{y_2} }{y_1} }{y} + ...,
\end{split}
\end{equation}
and hence
\begin{equation}
\begin{split}
    &\left| \inty{0}{x}{ U_E(x-y) \tilde{H} \inty{0}{y}{ U_E(y-y_1) \vec{u}_z(y_1) }{y_1} }{y} \right| \\
    &\leq x^2 \left( C^* ( e^{k x} + 1 ) \right)^2 |z - E| + x^3 \left( C^* ( e^{k x} + 1 ) \right)^3 |z - E|^2 + ... \\
    &\leq 2 x^2 \left( C^* (e^{k x} + 1) \right)^2 |z - E|
\end{split}
\end{equation}
whenever $x C^* (e^{k x} + 1) |z - E| < \frac{1}{2}$. Bounding successive terms of \eqref{eq:analogous} similarly, we have that
\begin{equation}
\begin{split}
    & \left| \de_z \vec{u}_z(x) \right| \leq \\
    & \hspace{.25cm} 2 x C^* (e^{k x} + 1) + 2 x^2 \left( C^* (e^{k x} + 1) \right)^2 |z - E| + 2 x^3 \left( C^* (e^{k x} + 1) \right)^3 |z - E|^2 + ... \\
    & \hspace{1cm} \leq 4 x C^* ( e^{k x } + 1 )
\end{split}
\end{equation}
whenever $x C^* (e^{k x} + 1) |z - E| < \frac{1}{2}$, as desired.

\section{Proof of Corollary \ref{cor:imag_asymp}} \label{sec:Theta}

In this section we prove Corollary \ref{cor:imag_asymp}. By definition we have that
\begin{equation}
    \frac{\Theta(E)}{\de_z \Theta(E)} = \frac{ u_E'(M) - i \sqrt{E} u_E(M) }{ \left. \de_z u_z'(M) \right|_{z = E} - i \left( \frac{1}{2 \sqrt{E}} \right) u_E(M) - i \sqrt{E} \left. \de_z u_z(M) \right|_{z = E} }.
\end{equation}
From equation \eqref{eq:dis} and Lemma \ref{lem:bound_uz} we have
\begin{equation} \label{eq:defff}
    \frac{\Theta(E)}{\de_z \Theta(E)} = \frac{ u_E'(M) - i \sqrt{E} u_E(M) }{ - \left( \inty{0}{M}{ u_E^2 (y) }{y} \right) \left( v_E'(M) - i \sqrt{E} v_E(M) \right) + O(M e^{- k M}) }.
\end{equation}
Since the denominator grows exponentially (recall equations \eqref{eq:nex}-\eqref{eq:v_bd_bel}) and using Lemma \ref{lem:bound_uz} we have from Taylor's theorem that
\begin{equation} \label{eq:defff_2}
    \frac{\Theta(E)}{\de_z \Theta(E)} = - \left( \inty{0}{M}{ u_E^2 (y) }{y} \right)^{-1} \frac{ u_E'(M) - i \sqrt{E} u_E(M) }{ v_E'(M) - i \sqrt{E} v_E(M) } + O(M e^{- 4 k M}).
\end{equation}
An elementary calculation shows that
\begin{equation}
\begin{split}
    &\frac{ u_E'(M) - i \sqrt{E} u_E(M) }{ v_E'(M) - i \sqrt{E} v_E(M) } = \\ 
    &\frac{ u_E'(M) v_E'(M) + E u_E(M) v_E(M) - i \sqrt{E} \left( u_E(M) v_E'(M) - u_E'(M) v_E(M) \right) }{ ( v_E'(M) )^2 + E ( v_E(M) )^2 },
\end{split}
\end{equation}
where we recognize the imaginary part in the numerator as the Wronskian of $u_E(x)$ and $v_E(x)$ evaluated at $M$, which equals $1$. We conclude finally that
\begin{equation}
    \frac{ \Theta(E) }{ \de_z \Theta(E) } = \frac{ i \sqrt{E} - \left[ u_E'(M) v_E'(M) + E u_E(M) v_E(M) \right] }{ \left[ ( v_E'(M) )^2 + E ( v_E(M) )^2 \right] \inty{0}{M}{ u_E^2 (y) }{y} } + O(M e^{- 4 k M}).
\end{equation}

\section{Generalization of our main result when parity symmetry does not necessarily hold} \label{sec:generalization_noparity}

In this section we present a generalization of the result presented in Section \ref{sec:main_result} which covers the case where the parity symmetry assumption (Assumption \ref{as:parity}) does not hold. Although many of the ideas presented in the previous sections, particularly the use of Floquet theory, carry over to this setting, there are significant differences between the results and proofs. We review the parts of the proof which are significantly different in this case in Section \ref{sec:gen_no_parity}.

We again consider the one-dimensional Schr\"odinger operator \eqref{eq:Schro}, assuming the potential $V$ is smooth and 1-periodic outside the interval $[-\rho,\rho]$ for some $\rho > 0$ (Assumption \ref{as:V_assump}). We again assume the existence of a bound state $\Phi(x)$ with positive eigenvalue $E > 0$ \eqref{eq:bound}. When $\Phi(0) \neq 0$ (note that after multiplying by a constant we can assume $\Phi(0) = 1$), define $u_{\vec{\zeta}}(x) \in C^\infty([-M,M])$ and $v_{\vec{\zeta}}(x) \in C^\infty([-M,M])$ as the solutions of
\begin{equation} \label{eq:eq_nosymmetry_phi_non-zero}
\begin{split}
    &( D_x^2 + V(x) - z ) u_{\vec{\zeta}} = 0, \quad u_{\vec{\zeta}}(0) = 1, u_{\vec{\zeta}}'(0) = w    \\
    &( D_x^2 + V(x) - z ) v_{\vec{\zeta}} = 0, \quad v_{\vec{\zeta}}(0) = \frac{ - w }{ 1 + w^2 }, v_{\vec{\zeta}}'(0) = \frac{ 1 }{ 1 + w^2 },
\end{split}
\end{equation}
for arbitrary $\vec{\zeta} = (w,z) \in \field{C}^2$, and let $\vec{\eta} = (\Phi'(0),E)$. Note that with these definitions $u_{\vec{\zeta}}$ and $v_{\vec{\zeta}}$ form a fundamental solution set whose Wronskian is $1$ and $u_{\vec{\eta}}(x) = \Phi(x)$ for $|x| \leq M$. When $\Phi(0) = 0$ (note that after multiplying by a constant we can assume $\Phi'(0) = 1$), define $u_{\vec{\zeta}}(x)$ and $v_{\vec{\zeta}}(x)$ by
\begin{equation} \label{eq:eq_nosymmetry_phiprime_non-zero}
\begin{split}
    &( D_x^2 + V(x) - z ) u_{\vec{\zeta}} = 0, \quad u_{\vec{\zeta}}(0) = -w, u_{\vec{\zeta}}'(0) = 1    \\
    &( D_x^2 + V(x) - z ) v_{\vec{\zeta}} = 0, \quad v_{\vec{\zeta}}(0) = \frac{ - 1 }{ 1 + w^2 }, v_{\vec{\zeta}}'(0) = \frac{ - w }{ 1 + w^2 },
\end{split}
\end{equation}
and let $\vec{\eta} = (0,E)$ (see Remark \ref{rem:convention} for an explanation of the convention we choose in \eqref{eq:eq_nosymmetry_phiprime_non-zero}). Again, $u_{\vec{\zeta}}$ and $v_{\vec{\zeta}}$ form a fundamental solution set and $u_{\vec{\eta}}(x) = \Phi(x)$ for $|x| \leq M$. Since $\Phi(x)$ is a bound state, via Floquet theory we have that $u_{\vec{\eta}}$ must exponentially decay and $v_{\vec{\eta}}$ must exponentially grow. 
We define
\begin{equation} \label{eq:X_again}
    ( X^\pm_1(\vec{\zeta}) , X^\pm_2(\vec{\zeta}) ) := \left( u_{\vec{\zeta}}(\pm M), u_{\vec{\zeta}}'(\pm M) \right),
\end{equation}
and let 
\begin{equation} \label{eq:thet_again}
    \Theta^\pm(\vec{\zeta}) := X^\pm_2(\vec{\zeta}) \mp i \sqrt{z} X^\pm_1(\vec{\zeta}).
\end{equation}
With the same conventions on the square root as in the case where parity symmetry holds (see below \eqref{eq:theta}), we have that when $E > 0$, $z^*$ is a resonance of $H_{\rm trunc}$ if and only if 
\begin{equation} \label{eq:condition}
    \vec{\Theta}(\vec{\zeta}^*) = 0, \text{ where } \vec{\Theta}(\vec{\zeta}) := \begin{pmatrix} \Theta^+(\vec{\zeta}) \\ \Theta^-(\vec{\zeta}) \end{pmatrix},
\end{equation}
where $\vec{\zeta}^* = (w^*,z^*)$ for some $w^* \in \field{C}$. 

In this setting, the counterpart of Theorem \ref{th:main_theorem} is
\begin{theorem} \label{th:main_theorem_noparity}
Let $V(x)$ satisfy Assumption \ref{as:V_assump}, let $\rho > 0$ be as in \eqref{eq:eventual_per}, and let $\Phi$ be a bound state with eigenvalue $E > 0$ as in \eqref{eq:bound}. Then there exist $k > 0$ and an $M_0 > \rho \geq 0$ such that for all $M \geq M_0$,
\begin{enumerate}
\item $\vec{\Theta}(\vec{\zeta})$ given by \eqref{eq:condition} has a unique root $\vec{\zeta}^*$ in the ball
\begin{equation}
    \Omega_M = \left\{ \vec{\zeta} : | \vec{\zeta} - \vec{\eta} | \leq \frac{ 1 }{ M^2 } e^{- k M} \right\}.
\end{equation}
\item The location of the root $\vec{\zeta}^*$ can be precisely characterized as
\begin{equation} \label{eq:res_expansion}
    \vec{\zeta}^* = \vec{\eta} - \Xi \vec{\Theta}(\vec{\eta}) + O(e^{- 4 k M}),
\end{equation}
where $\Xi$ is the matrix
\begin{equation} \label{eq:Xi}
    \Xi := \frac{1}{\mathcal{N}(\vec{\eta})} \begin{pmatrix} \de_z \Theta^-(\vec{\eta}) & - \de_z \Theta^+(\vec{\eta}) \\ - \de_w \Theta^-(\vec{\eta}) & \de_w \Theta^+(\vec{\eta}) \end{pmatrix}, \\
\end{equation}
where $\mathcal{N} (\vec{\eta}) := \de_w \Theta^+(\vec{\eta}) \de_z \Theta^-(\vec{\eta}) - \de_z \Theta^+(\vec{\eta}) \de_w \Theta^-(\vec{\eta})$.
\end{enumerate}
\end{theorem}
The overall idea of the proof of Theorem \ref{th:main_theorem_noparity} is the same as that of Theorem \ref{th:main_theorem}. We give a sketch of the proof here, postponing a discussion of the details which differ from the proof of Theorem \ref{th:main_theorem} to Section \ref{sec:gen_no_parity}. Define $\Psi: \field{C}^2 \rightarrow \field{C}^2$ by
\begin{equation} \label{eq:Psi_2}
    \Psi(\vec{\zeta}) := \vec{\zeta} - \Xi \vec{\Theta}(\vec{\zeta}),
\end{equation}
where $\Xi$ is defined by \eqref{eq:Xi}. Assuming $\mathcal{N}(\vec{\eta}) \neq 0$, then $\text{det } \Xi = \frac{1}{[\mathcal{N}(\vec{\eta})]^3}$ so that $\Xi$ is invertible, and hence 
\begin{equation}
    \Psi(\vec{\zeta}) = \vec{\zeta} \iff \vec{\Theta}(\vec{\zeta}) = 0.
\end{equation}
It remains to show that $\Psi$ is a contraction in the ball
\begin{equation} \label{eq:omega_M_2}
    \Omega_M := \left\{ \vec{\zeta} \in \field{C}^2 : | \vec{\zeta} - \vec{\eta} | < \frac{1}{M^2} e^{- k M} \right\},
\end{equation}
where $k$ characterizes the exponential decay of the bound state just as in the proof of Theorem \ref{th:main_theorem} (recall Section \ref{sec:bound_de_z_thet}). Parts (1) and (2) of the theorem then follow from the Banach fixed point theorem and the asymptotic formula for the fixed point as $\lim_{n \rightarrow \infty} \Psi^n(\vec{\eta})$ respectively. For the proofs that $\mathcal{N}(\vec{\eta}) \neq 0$ and that $\Psi$ defined by \eqref{eq:Psi_2} with $\Xi$ as in \eqref{eq:Xi} is a contraction in $\Omega_M$, see Section \ref{sec:gen_no_parity}.

\begin{remark}
The matrix $\Xi$ is the inverse of
\begin{equation}
    \begin{pmatrix} \de_w \Theta^+(\vec{\eta}) & \de_z \Theta^+(\vec{\eta}) \\ \de_w \Theta^-(\vec{\eta}) & \de_z \Theta^-(\vec{\eta}) \end{pmatrix},
\end{equation}
which is the Jacobian of the map
\begin{equation}
    \vec{\zeta} \mapsto \left( \Theta^+(\vec{\zeta}) , \Theta^-(\vec{\zeta}) \right)
\end{equation}
evaluated at $\vec{\zeta} = \vec{\eta}$. This is consistent with the parity symmetry case (compare \eqref{eq:res_expansion} with \eqref{eq:asymp}), where $\frac{1}{\de_z \Theta(E)}$ is the inverse of $\de_z \Theta(E)$, which is the Jacobian of the map
\begin{equation}
    z \mapsto \Theta(z)
\end{equation}
evaluated at $z = E$.
\end{remark}

Just as in Corollary \ref{cor:imag_asymp}, we can obtain an expression for the first order correction appearing in \eqref{eq:res_expansion} in terms of $u_{\vec{\zeta}}$ and $v_{\vec{\zeta}}$ (the solutions of \eqref{eq:eq_nosymmetry_phi_non-zero}). This expression implies the following asymptotics for $w^*$ and $z^*$, where $\vec{\zeta^*} = (w^*,z^*)$:
\begin{corollary} \label{cor:correction_noparity}
Let $\vec{\zeta^*} = (w^*,z^*)$ be as in \eqref{eq:res_expansion}. Then 
\begin{equation} \label{eq:w_star}
\begin{split}
    w^* &= w_0 - \left[ \inty{-M}{0}{ u_{\vec{\eta}}^2(y) }{y} \left( v'_{\vec{\eta}}(-M) + i \sqrt{E} v_{\vec{\eta}}(-M) \right) \left( u_{\vec{\eta}}'(M) - i \sqrt{E} u_{\vec{\eta}}(M) \right) \right.   \\
    &\phantom{=++} \left. + \inty{0}{M}{ u_{\vec{\eta}}^2(y) }{y} \left( v'_{\vec{\eta}}(M) - i \sqrt{E} v_{\vec{\eta}}(M) \right) \left( u_{\vec{\eta}}'(-M)  + i \sqrt{E} u_{\vec{\eta}}(-M) \right) \right]  \\
    &\phantom{=++} \times \left[ \inty{-M}{M}{ u_{\vec{\eta}}^2(y) }{y} \left( v_{\vec{\eta}}'(M) - i \sqrt{E} v_{\vec{\eta}}(M) \right) \left( v_{\vec{\eta}}'(-M) + i \sqrt{E} v_{\vec{\eta}}(-M) \right) \right]^{-1} \\
    &\phantom{=+++++} + O(e^{- 4 k M}),
\end{split}
\end{equation}
where $w_0 = \Phi'(0)$ when $\Phi(0) \neq 0$, $w_0 = 0$ when $\Phi(0) = 0$, and
\begin{equation} \label{eq:z_star}
\begin{split}
    z^* &= E - \left[ \left( v_{\vec{\eta}}'(M) - i \sqrt{E} v_{\vec{\eta}}(M) \right) \left( u_{\vec{\eta}}'(-M) + i \sqrt{E} u_{\vec{\eta}}(-M) \right) \right.     \\
    &\phantom{=++++} \left. - \left( v_{\vec{\eta}}'(-M) + i \sqrt{E} v_{\vec{\eta}}(-M) \right) \left( u_{\vec{\eta}}'(M) - i \sqrt{E} u_{\vec{\eta}}(M) \right) \right] \\
    &\phantom{=++} \times \left[ \inty{-M}{M}{ u_{\vec{\eta}}^2(y) }{y} \left( v_{\vec{\eta}}'(M) - i \sqrt{E} v_{\vec{\eta}}(M) \right) \left( v_{\vec{\eta}}'(-M) + i \sqrt{E} v_{\vec{\eta}}(-M) \right) \right]^{-1} \\
    &\phantom{=++++} + O(e^{- 4 k M}).
\end{split}
\end{equation}
\end{corollary}
\begin{proof} See Section \ref{sec:proof_correction_noparity}.
\end{proof}

The consequences of Theorem \ref{th:main_theorem_noparity} and Corollary \ref{cor:correction_noparity} can be summarized as follows:

\begin{corollary}
Let $V, \Phi, E > 0, \rho > 0, u_{\vec{\zeta}}$, and $v_{\vec{\zeta}}$ be as in Theorem \ref{th:main_theorem_noparity} and Corollary \ref{cor:correction_noparity}. Then there exist $k > 0$ and an $M_0 > \rho \geq 0$ such that for all $M \geq M_0$, $H_{\rm trunc}$ has a resonance $z^*$ given up to corrections of $O(e^{- 4 k M})$ by \eqref{eq:z_star}. The associated resonant state equals
\begin{equation}
    \Phi^*(x) = \begin{cases} u_{\vec{\zeta^*}}(-M) e^{- i \sqrt{z^*}(x - M)} & x \leq -M \\ u_{\vec{\zeta^*}}(x) & |x| \leq M \\ u_{\vec{\zeta^*}}(M) e^{i \sqrt{z^*}(x - M)} & x \geq M, \end{cases}
\end{equation}
where $\vec{\zeta^*} = (w^*,z^*)$, and $w^*$ is given up to corrections of $O(e^{- 4 k M})$ by \eqref{eq:w_star}.
\end{corollary}

\begin{remark} \label{rem:reduction_to_parity}
It is important to check that the results of the present section reduce to those of Section \ref{sec:main_result} when parity symmetry holds. First, as our results in Section \ref{sec:main_result} confirm, the resonant state whose resonance is nearby in the complex plane to $E$ must be even (resp. odd) whenever the bound state is even (resp. odd). Hence the correction to $w_0$ in \eqref{eq:w_star} should vanish when parity symmetry holds. Second, the correction term in \eqref{eq:z_star} should reduce to \eqref{eq:correction} when parity symmetry holds. We confirm these expectations in Section \ref{sec:reduction_when_parity_holds}.
\end{remark}

\begin{remark}
We remark on further generalizations of our result which should follow from an essentially identical analysis. First, recall that we assume the structure is truncated at $x = \pm M$ \eqref{eq:V_trunc}. Our result could be generalized to the case where the structure is truncated at some $M_+ > 0$ and $M_- < 0$ where $M_+ \neq - M_-$, in the limit where $M_\pm \rightarrow \pm \infty$. Second, our result should even generalize to the case of differing periodic structures either side of the defect region, even with differing periods. Since extending our results to these settings would complicate the statement of our results and would not require substantial modifications of the proof, we do not consider these cases here.
\end{remark}

\section{Proof of Theorem \ref{th:main_theorem_noparity}: derivation of $\Xi$ and proof $\mathcal{N}(\vec{\eta}) \neq 0$} \label{sec:gen_no_parity}

In this section we present the parts of the proof of Theorem \ref{th:main_theorem_noparity} which differ substantially from the proof of Theorem \ref{th:main_theorem}. Specifically, we show that the form of $\Xi$ \eqref{eq:Xi} makes the map $\Psi$ \eqref{eq:Psi_2} a contraction in the ball $\Omega_M$ \eqref{eq:omega_M_2}, and prove that $\mathcal{N}(\vec{\eta}) \neq 0$ so that $\Xi$ is well-defined.

We will first show that the form of $\Xi$ \eqref{eq:Xi} makes $\Psi$ a contraction in the ball $\Omega_M$. Writing $\Psi(\vec{\zeta}) = ( \Psi_1(\vec{\zeta}),\Psi_2(\vec{\zeta}) )^\top$, by Taylor's theorem we have that $\Psi$ is a contraction if every element of the Jacobian matrix 
\begin{equation}
    J := \begin{pmatrix} \de_w \Psi_1 & \de_z \Psi_1 \\ \de_w \Psi_2 & \de_z \Psi_2 \end{pmatrix}
\end{equation}
can be bounded uniformly by $\frac{1}{2}$ in $\Omega_M$. Recalling the form of $\Psi$ \eqref{eq:Psi_2}, writing $\Xi$ as
\begin{equation}
    \Xi = \begin{pmatrix} \Xi_{11} & \Xi_{12} \\ \Xi_{21} & \Xi_{22} \end{pmatrix},
\end{equation}
and assuming $\Xi$ is independent of $\vec{\zeta}$, we have that 
\begin{equation}
    J = \begin{pmatrix} J_{11} & J_{12} \\ J_{21} & J_{22} \end{pmatrix} = \begin{pmatrix} 1 - \Xi_{11} \de_w \Theta^+ - \Xi_{12} \de_w \Theta^- & - \Xi_{11} \de_z \Theta^+ - \Xi_{12} \de_z \Theta^- \\ - \Xi_{21} \de_w \Theta^+ - \Xi_{22} \de_w \Theta^- & 1 - \Xi_{21} \de_z \Theta^+ - \Xi_{22} \de_z \Theta^- \end{pmatrix}.
\end{equation}
Substituting the form of $\Xi$ given in \eqref{eq:Xi} (we will prove in Lemma \ref{lem:bound_N_below} that $\mathcal{N}(\vec{\eta}) \neq 0$ so that $\Xi$ is well-defined), the diagonal entries of $J$ are
\begin{equation}
\begin{split}
    &J_{11} = 1 - \frac{ \de_w \Theta^+(\vec{\zeta}) \de_z \Theta^-(\vec{\eta}) - \de_z \Theta^+(\vec{\eta}) \de_w \Theta^-(\vec{\zeta}) }{ \mathcal{N}(\vec{\eta}) }, \\
    &J_{22} = 1 - \frac{ \de_w \Theta^+(\vec{\eta}) \de_z \Theta^-(\vec{\zeta}) - \de_z \Theta^+(\vec{\zeta}) \de_w \Theta^-(\vec{\eta}) }{\mathcal{N} (\vec{\eta})},
\end{split}
\end{equation}
while the off-diagonal terms are
\begin{equation}
\begin{split}
    &J_{12} = - \frac{ \de_z \Theta^+(\vec{\zeta}) \de_z \Theta^-(\vec{\eta}) - \de_z \Theta^+(\vec{\eta}) \de_z \Theta^-(\vec{\zeta}) }{ \mathcal{N}(\vec{\eta}) } \\
    &J_{21} = - \frac{ \de_w \Theta^+(\vec{\zeta}) \de_w \Theta^-(\vec{\eta}) - \de_w \Theta^+(\vec{\eta}) \de_w \Theta^-(\vec{\zeta}) }{ \mathcal{N}(\vec{\eta}) }.
\end{split}
\end{equation}
To see that this choice of $\Xi$ makes the entries of $J$ small for $z \in \Omega_M$ and $M$ sufficiently large, note that we can re-write the diagonal entries of $J$ as
\begin{equation} \label{eq:diagonal}
\begin{split}
    &J_{11} = \frac{1}{\mathcal{N} (\vec{\eta})} \left[ \left( \de_w \Theta^+(\vec{\eta}) - \de_w \Theta^+(\vec{\zeta}) \right) \de_z \Theta^-(\vec{\eta}) - \de_z \Theta^+(\vec{\eta}) \left( \de_w \Theta^-(\vec{\eta}) - \de_w \Theta^-(\vec{\zeta}) \right) \right]    \\
    &J_{22} = \frac{1}{\mathcal{N} (\vec{\eta})} \left[ \de_w \Theta^+(\vec{\eta}) \left( \de_z \Theta^-(\vec{\eta}) - \de_z \Theta^-(\vec{\zeta}) \right) - \left( \de_z \Theta^+(\vec{\eta}) - \de_z \Theta^+(\vec{\zeta}) \right) \de_w \Theta^-(\vec{\eta}) \right],  \\
\end{split}
\end{equation}
and the off-diagonal terms as 
\begin{equation} \label{eq:offdiagonal}
\begin{split}
    &J_{12} = \frac{1}{\mathcal{N} (\vec{\eta})} \left[ \de_z \Theta^+(\vec{\eta}) \left( \de_z \Theta^-(\vec{\zeta}) - \de_z \Theta^-(\vec{\eta}) \right) \right. \\
    & \hspace{2cm} \left. + \left( \de_z \Theta^+(\vec{\eta}) - \de_z \Theta^+(\vec{\zeta}) \right) \de_z \Theta^-(\vec{\eta}) \right] ,  \\
    &J_{21} = \frac{1}{\mathcal{N} (\vec{\eta})} \left[ \de_w \Theta^+(\vec{\eta}) \left( \de_w \Theta^-(\vec{\zeta}) - \de_w \Theta^-(\vec{\eta}) \right)  \right. \\
    & \hspace{2cm} \left. - \left( \de_w \Theta^+(\vec{\zeta}) - \de_w \Theta^+(\vec{\eta}) \right) \de_w \Theta^-(\vec{\eta}) \right].
\end{split}
\end{equation}
We see that every component of $J$ is a sum of two terms. We will show that each component can be bounded by $\frac{1}{2}$ by showing that each of these terms can be bounded by $\frac{1}{4}$. The proofs follow a similar logic to that given in Section \ref{sec:fixed}. To give the idea, we first consider a representative term before making the necessary estimates precise. We aim to prove that, for example,
\begin{equation} \label{eq:est_to_prove}
    \left| \frac{1}{\mathcal{N} (\vec{\eta})} \left( \de_z \Theta^+(\vec{\eta}) - \de_z \Theta^+(\vec{\zeta}) \right) \de_w \Theta^-(\vec{\eta}) \right| < \frac{1}{4},
\end{equation}
or equivalently
\begin{equation}
   \left| \de_z \Theta^+(\vec{\eta}) - \de_z \Theta^+(\vec{\zeta}) \right| \left| \de_w \Theta^-(\vec{\eta}) \right| < \frac{1}{4} \left| \mathcal{N}(\vec{\eta}) \right|,
\end{equation}
for all $z \in \Omega_M$. By the mean value theorem, we have that
\begin{equation}
    \left| \de_z \Theta^+(\vec{\eta}) - \de_z \Theta^+(\vec{\zeta}) \right| \leq | \vec{\zeta} - \vec{\eta} | \sup_{z \in \Omega_M} | \de_z^2 \Theta^+(\vec{\zeta}) |
\end{equation}
for $\vec{\zeta} \in \Omega_M$. Since for $\vec{\zeta} \in \Omega_M$, $|\vec{\zeta} - \vec{\eta}|$ is exponentially small in $M$, estimate \eqref{eq:est_to_prove} is proved if we can show that
\begin{equation} \label{eq:necessary}
    \sup_{z \in \Omega_M} | \de_z^2 \Theta^+(\vec{\zeta}) | \leq C e^{k M}, \left| \de_w \Theta^-(\vec{\eta}) \right| \leq C e^{k M}, \text{ and } \left| \mathcal{N}(\vec{\eta}) \right| \geq C e^{2 k M}
\end{equation}
for constants $k > 0$ and $C > 0$. It is clear that with analogous bounds as in \eqref{eq:necessary}, the other terms appearing in \eqref{eq:diagonal}-\eqref{eq:offdiagonal} can be bounded along the same lines.

Now that the basic idea has been established, we make precise the lemmas which are necessary to prove that $\Psi$ is a contraction. First, we require a lemma analogous to Lemma \ref{lem:bound_de_z_thet}. 
\begin{lemma} \label{lem:bound_N_below}
Let $\Theta^\pm(\vec{\zeta})$ be as in \eqref{eq:thet_again}. There exist positive constants $C$ and $k$ and $M_0 > \rho > 0$ such that for all $M \geq M_0$,
\begin{equation} \label{eq:N_formula}
    \mathcal{N}(\vec{\eta}) := \de_w \Theta^+(\vec{\eta}) \de_z \Theta^-(\vec{\eta}) - \de_z \Theta^+(\vec{\eta}) \de_w \Theta^-(\vec{\eta})
\end{equation}
satisfies
\begin{equation} \label{eq:bound_bel_N}
    \left| \mathcal{N}(\vec{\eta}) \right| \geq C e^{2 k M}.
\end{equation}
\end{lemma}
Second, we require a lemma analogous to Lemma \ref{lem:bound_de_z_2_thet}.
\begin{lemma} \label{lem:bound_above}
Let $k$ be as in Lemma \ref{lem:bound_N_below} and $\Omega_M$ be as in \eqref{eq:omega_M_2}. Then there exist positive constants $C$ and $M_0 > \rho > 0$ such that for all $M \geq M_0$ the following estimates hold
\begin{equation} \label{eq:de_z_bd}
    | \de_z \Theta^\pm(\vec{\eta}) | \leq C e^{k M}, | \de_w \Theta^\pm(\vec{\eta}) | \leq C e^{k M}, 
\end{equation}
\begin{equation} \label{eq:de_2_z_bd}
    \sup_{\vec{\zeta} \in \Omega_M} | \de_z^2 \Theta^\pm(\vec{\zeta}) | \leq C e^{k M}, \sup_{\vec{\zeta} \in \Omega_M} | \de_w^2 \Theta^\pm(\vec{\zeta}) | \leq C e^{k M}.
\end{equation}
\end{lemma}
We explain the important points of these proofs, particularly those points which differ from the proofs of Lemmas \ref{lem:bound_de_z_thet} and \ref{lem:bound_de_z_2_thet}, in the following section.

\subsection{Proofs of Lemmas \ref{lem:bound_N_below} and \ref{lem:bound_above}} \label{sec:proofs_of_lems}

Using the definitions of $\Theta^\pm(\vec{\eta})$ \eqref{eq:X_again}-\eqref{eq:thet_again} we have that
\begin{equation} \label{eq:all_de_thet}
\begin{split}
    \de_z \Theta^\pm(\vec{\zeta}) &= \left. \de_z u_{\vec{\zeta}}'(\pm M) \right|_{\vec{\zeta}=\vec{\eta}}  \mp i \sqrt{E} \left. \de_z u_{\vec{\zeta}}(\pm M) \right|_{\vec{\zeta}=\vec{\eta}}  \mp i \frac{1}{2 \sqrt{E}} \left. u_{\vec{\zeta}}(\pm M) \right|_{\vec{\zeta}=\vec{\eta}}  \\
    \de_w \Theta^\pm(\vec{\zeta}) &= \left. \de_w u_{\vec{\zeta}}'(\pm M) \right|_{\vec{\zeta}=\vec{\eta}} \mp i \sqrt{E} \left. \de_w u_{\vec{\zeta}}(\pm M) \right|_{\vec{\zeta}=\vec{\eta}}.
\end{split}
\end{equation}
Suppose that $\Phi(0) \neq 0$ so that $u_{\vec{\zeta}}$ and $v_{\vec{\zeta}}$ are defined by \eqref{eq:eq_nosymmetry_phi_non-zero}. Differentiating the equation for $u_{\vec{\zeta}}$ we have that $\de_w u_{\vec{\zeta}}(x)$ satisfies
\begin{equation}
    \left( D_x^2 + V(x) - z \right) \de_w u_{\vec{\zeta}} = 0, \quad \de_w u_{\vec{\zeta}}(0) = 0, \de_w u_{\vec{\zeta}}'(0) = 1,
\end{equation}
which has the unique solution
\begin{equation} \label{eq:de_w_u}
    \de_w u_{\vec{\zeta}}(x) = \frac{w}{1 + w^2} u_{\vec{\zeta}}(x) + v_{\vec{\zeta}}(x).
\end{equation}
If instead $\Phi(0) = 0$, so that $u_{\vec{\zeta}}$ and $v_{\vec{\zeta}}$ are defined by \eqref{eq:eq_nosymmetry_phiprime_non-zero}, then we have that $\de_w u_{\vec{\zeta}}(x)$ satisfies
\begin{equation}
    \left( D_x^2 + V(x) - z \right) \de_w u_{\vec{\zeta}} = 0, \quad \de_w u_{\vec{\zeta}}(0) = 1, \de_w u_{\vec{\zeta}}'(0) = 0,
\end{equation}
which has the unique solution
\begin{equation} \label{eq:de_w_u_2}
    \de_w u_{\vec{\zeta}}(x) = \frac{ w }{1 + w^2} u_{\vec{\zeta}}(x) + v_{\vec{\zeta}}(x).
\end{equation}
\begin{remark} \label{rem:convention}
Note that by virtue of the convention we chose for $u_{\vec{\zeta}}$ and $v_{\vec{\zeta}}$ in \eqref{eq:eq_nosymmetry_phiprime_non-zero}, the formulas for $\de_w u_{\vec{\zeta}}$ in \eqref{eq:de_w_u} and \eqref{eq:de_w_u_2} match. This allows us to write formulas which are valid in both cases without modification.
\end{remark}
By variation of parameters as in Section \ref{sec:ests_2} we have that 
\begin{equation} \label{eq:de_z_u}
    \de_z u_{\vec{\zeta}}(x) = \left( \inty{0}{x}{ v_{\vec{\zeta}}(y) u_{\vec{\zeta}}(y) }{y} \right) u_{\vec{\zeta}}(x) - \left( \inty{0}{x}{ u_{\vec{\zeta}}^2(y) }{y} \right) v_{\vec{\zeta}}(x).
\end{equation}
We now argue as in Section \ref{sec:bound_de_z_thet}. Setting $\vec{\zeta} = \vec{\eta}$, by Floquet theory and the assumption that $u_{\vec{\eta}}(x) = \Phi(x)$ for $|x| \leq M$, we have that $u_{\vec{\eta}}(x)$ exponentially decays in $|x|$, i.e. that there exist constants $C > 0$ and $k > 0$ such that
\begin{equation} \label{eq:exp_dec}
    | u_{\vec{\eta}}(x) | \leq C e^{- k |x|}, \quad | u'_{\vec{\eta}}(x) | \leq C e^{- k |x|},
\end{equation}
and
\begin{equation} \label{eq:exp_gro}
    | v_{\vec{\eta}}(x) | \leq C e^{k |x|}, \quad | v'_{\vec{\eta}}(x) | \leq C e^{k |x|}.
\end{equation}
We can now give the proof of Lemma \ref{lem:bound_above}.
\begin{proof}[Proof of Lemma \ref{lem:bound_above}]
The estimates \eqref{eq:de_z_bd} follow immediately from substituting formulas \eqref{eq:de_w_u}, \eqref{eq:de_w_u_2}, and \eqref{eq:de_z_u} into \eqref{eq:all_de_thet} and applying the estimates \eqref{eq:exp_dec}-\eqref{eq:exp_gro}. The estimates \eqref{eq:de_2_z_bd} follow by the same argument as in Section \ref{sec:bound_de_z_2_thet}.  
\end{proof}
It remains to prove Lemma \ref{lem:bound_N_below}. 
\begin{proof}[Proof of Lemma \ref{lem:bound_N_below}]
Directly setting $\vec{\zeta} = \vec{\eta}$ and substituting \eqref{eq:de_w_u}, \eqref{eq:de_w_u_2}, and \eqref{eq:de_z_u} into \eqref{eq:N_formula} yields a long expression for $\mathcal{N}(\vec{\eta})$. To prove the bound below \eqref{eq:bound_bel_N}, we note that the terms in this expression which depend quadratically on $v_{\vec{\eta}}(x)$ will dominate for sufficiently large $M$ and hence we can ignore the other terms. 

Regardless of whether $\Phi(0) \neq 0$ or $\Phi(0) = 0$, the terms which depend quadratically on $v_{\vec{\eta}}$ are
\begin{equation}
\begin{split}
    &\left( v_{\vec{\eta}}'(M) - i \sqrt{E} v_{\vec{\eta}}(M) \right) \\
    & \times \left( \left( - \inty{0}{-M}{ u_{\vec{\eta}}^2(y) }{y} \right) v_{\vec{\eta}}'(-M) + i \sqrt{E} \left( - \inty{0}{-M}{ u_{\vec{\eta}}^2(y) }{y} \right) v_{\vec{\eta}}(-M) \right)    \\
    &- \left( \left( - \inty{0}{M}{ u_{\vec{\eta}}^2(y) }{y} \right) v_{\vec{\eta}}'(M) - i \sqrt{E} \left( - \inty{0}{M}{ u_{\vec{\eta}}^2(y) }{y} \right) v_{\vec{\eta}}(M) \right)  \\
    & \hspace{2cm} \times \left( v_{\vec{\eta}}'(-M) + i \sqrt{E} v_{\vec{\eta}}(-M) \right),
\end{split}
\end{equation}
which simplifies to
\begin{equation}
\begin{split}
    &\inty{-M}{M}{ u_{\vec{\eta}}^2(y) }{y} \left( v_{\vec{\eta}}'(M) v_{\vec{\eta}}'(-M) - i \sqrt{E} v_{\vec{\eta}}(M) v_{\vec{\eta}}'(-M) \right. \\
    & \hspace{2cm}  \left. + i \sqrt{E} v_{\vec{\eta}}'(M) v_{\vec{\eta}}(-M) + E v_{\vec{\eta}}(M) v_{\vec{\eta}}(-M) \right)    \\
    &= \inty{-M}{M}{ u_{\vec{\eta}}^2(y) }{y} \left( v_{\vec{\eta}}'(M) - i \sqrt{E} v_{\vec{\eta}}(M) \right) \left( v_{\vec{\eta}}'(-M) + i \sqrt{E} v_{\vec{\eta}}(-M) \right).
\end{split}
\end{equation}
The estimate \eqref{eq:bound_bel_N} now follows since $\inty{-M}{M}{ u_{\vec{\eta}}^2(y) }{y}$ can be bounded below uniformly in $M$, and the two factors depending on $v_{\vec{\eta}}$ can be bounded below by $C e^{k M}$ exactly as in the proof of Lemma \ref{lem:for_bd_bel} (note that both factors $\neq 0$ since $v_{\vec{\eta}}(x)$ is real, and see in particular equations \eqref{eq:exp_behavior}-\eqref{eq:v_bd_bel}).
\end{proof}

\section{Proof of Corollary \ref{cor:correction_noparity}} \label{sec:proof_correction_noparity}

The correction term to be simplified is as follows
\begin{equation} \label{eq:expression}
    \Xi \Theta(\vec{\eta}) = \frac{1}{\mathcal{N} (\vec{\eta})} \begin{pmatrix} \de_z \Theta^-(\vec{\eta}) \Theta^+(\vec{\eta}) - \de_z \Theta^+(\vec{\eta}) \Theta^-(\vec{\eta}) \\ \de_w \Theta^+(\vec{\eta}) \Theta^-(\vec{\eta}) - \de_w \Theta^-(\vec{\eta}) \Theta^+(\vec{\eta}) \end{pmatrix},
\end{equation}
where $\mathcal{N}(\vec{\eta}) = \de_w \Theta^+(\vec{\eta}) \de_z \Theta^-(\vec{\eta}) - \de_z \Theta^+(\vec{\eta}) \de_w \Theta^-(\vec{\eta})$. Whether $\Phi(0) \neq 0$ or $\Phi(0) = 0$, we substitute expressions \eqref{eq:all_de_thet}, \eqref{eq:de_w_u}, and \eqref{eq:de_z_u} into \eqref{eq:expression}. Just as in the proof of Lemma \ref{lem:bound_N_below} in Section \ref{sec:proofs_of_lems}, it suffices to consider the dominant terms i.e. those depending linearly or quadratically on $v_{\vec{\eta}}$. We have already seen from the proof of Lemma \ref{lem:bound_N_below} in Section \ref{sec:proofs_of_lems} that 
\begin{equation}
    \mathcal{N}(\vec{\eta}) = \inty{-M}{M}{ u_{\vec{\eta}}^2(y) }{y} \left( v_{\vec{\eta}}'(M) - i \sqrt{E} v_{\vec{\eta}}(M) \right) \left( v_{\vec{\eta}}'(-M) + i \sqrt{E} v_{\vec{\eta}}(-M) \right) + O(1).
\end{equation}
Similarly, ignoring terms which are independent of $v_{\vec{\eta}}$ gives
\begin{equation}
\begin{split}
    &\de_z \Theta^-(\vec{\eta}) \Theta^+(\vec{\eta}) - \de_z \Theta^+(\vec{\eta}) \Theta^-(\vec{\eta}) \\
    &= \inty{-M}{0}{ u_{\vec{\eta}}^2(y) }{y} \left( v'_{\vec{\eta}}(-M) + i \sqrt{E} v_{\vec{\eta}}(-M) \right) \left( u_{\vec{\eta}}'(M) - i \sqrt{E} u_{\vec{\eta}}(M) \right)    \\
    &\phantom{=} + \inty{0}{M}{ u_{\vec{\eta}}^2(y) }{y} \left( v'_{\vec{\eta}}(M) - i \sqrt{E} v_{\vec{\eta}}(M) \right) \left( u_{\vec{\eta}}'(-M) + i \sqrt{E} u_{\vec{\eta}}(-M) \right) + O(e^{- 2 k M}),
\end{split}
\end{equation}
and
\begin{equation}
\begin{split}
    &\de_w \Theta^+(\vec{\eta}) \Theta^-(\vec{\eta}) - \de_w \Theta^-(\vec{\eta}) \Theta^+(\vec{\eta})   \\
    &= \left( v_{\vec{\eta}}'(M) - i \sqrt{E} v_{\vec{\eta}}(M) \right) \left( u_{\vec{\eta}}'(-M) + i \sqrt{E} u_{\vec{\eta}}(-M) \right) \\
    &\phantom{=} - \left( v_{\vec{\eta}}'(-M) + i \sqrt{E} v_{\vec{\eta}}(-M) \right) \left( u_{\vec{\eta}}'(M) - i \sqrt{E} u_{\vec{\eta}}(M) \right) + O(e^{- 2 k M}).
\end{split}
\end{equation}
Putting everything together and writing $\Xi \Theta(\vec{\eta}) = ( (\Xi \Theta(\vec{\eta}) )_1, (\Xi \Theta(\vec{\eta}) )_2 )^\top$ we have 

\begin{equation} \label{eq:Xi_1}
\begin{split}
    \left( \Xi \Theta(\vec{\eta}) \right)_1 &= \left[ \inty{-M}{0}{ u_{\vec{\eta}}^2(y) }{y} \left( v'_{\vec{\eta}}(-M) + i \sqrt{E} v_{\vec{\eta}}(-M) \right) \left( u_{\vec{\eta}}'(M) - i \sqrt{E} u_{\vec{\eta}}(M) \right) \right.   \\
    &\left. + \inty{0}{M}{ u_{\vec{\eta}}^2(y) }{y} \left( v'_{\vec{\eta}}(M) - i \sqrt{E} v_{\vec{\eta}}(M) \right) \left( u_{\vec{\eta}}'(-M)  + i \sqrt{E} u_{\vec{\eta}}(-M) \right) \right]  \\
    & \times \left[ \inty{-M}{M}{ u_{\vec{\eta}}^2(y) }{y} \left( v_{\vec{\eta}}'(M) - i \sqrt{E} v_{\vec{\eta}}(M) \right) \left( v_{\vec{\eta}}'(-M) + i \sqrt{E} v_{\vec{\eta}}(-M) \right) \right]^{-1} \\
    &\phantom{=++} + O(e^{- 4 k M}),
\end{split}
\end{equation}
and 
\begin{equation} \label{eq:Xi_2}
\begin{split}
    \left( \Xi \Theta(\vec{\eta}) \right)_2 &= \left[ \left( v_{\vec{\eta}}'(M) - i \sqrt{E} v_{\vec{\eta}}(M) \right) \left( u_{\vec{\eta}}'(-M) + i \sqrt{E} u_{\vec{\eta}}(-M) \right) \right.     \\
    & \hspace{.75cm} \left. - \left( v_{\vec{\eta}}'(-M) + i \sqrt{E} v_{\vec{\eta}}(-M) \right) \left( u_{\vec{\eta}}'(M) - i \sqrt{E} u_{\vec{\eta}}(M) \right) \right] \\
    & \times \left[ \inty{-M}{M}{ u_{\vec{\eta}}^2(y) }{y} \left( v_{\vec{\eta}}'(M) - i \sqrt{E} v_{\vec{\eta}}(M) \right) \left( v_{\vec{\eta}}'(-M) + i \sqrt{E} v_{\vec{\eta}}(-M) \right) \right]^{-1} \\
    &\phantom{=++} + O(e^{- 4 k M}).
\end{split}
\end{equation}

\subsection{Verification that result of Corollary \ref{cor:correction_noparity} reduces to that of Corollary \ref{cor:imag_asymp} when parity symmetry holds} \label{sec:reduction_when_parity_holds}

Suppose that parity symmetry, i.e. Assumption \ref{as:parity}, holds. When the bound state $\Phi$ is even (resp. odd), we can assume that $u_{\vec{\eta}}$ and $v_{\vec{\eta}}$ are even (resp. odd) and odd (resp. even). Since the other case is identical, we consider only the case where $u_{\vec{\eta}}$ is even and $v_{\vec{\eta}}$ is odd. We have then that for all $x \in [0,M]$,
\begin{equation}
\begin{aligned}
    &u_{\vec{\eta}}(-x) = u_{\vec{\eta}}(x), & &v_{\vec{\eta}}(-x) = - v_{\vec{\eta}}(x),    \\
    &u_{\vec{\eta}}'(-x) = - u_{\vec{\eta}}'(x), & &v_{\vec{\eta}}'(-x) = v_{\vec{\eta}}'(x).
\end{aligned}
\end{equation}
Substituting these identities into \eqref{eq:Xi_1} and \eqref{eq:Xi_2} and repeating the simplifications of Section \ref{sec:Theta} we have that
\begin{equation}
    \Xi \Theta(\vec{\eta}) = \begin{pmatrix} 0 \\ \frac{ i \sqrt{E} - \left[ u_E'(M) v_E'(M) + E u_E(M) v_E(M) \right] }{ \left[ ( v_E'(M) )^2 + E ( v_E(M) )^2 \right] \inty{0}{M}{ u_E^2 (y) }{y} } \end{pmatrix} + O(M e^{- 4 k M}),
\end{equation}
as desired.

\section{Further generalizations} \label{sec:generalizations} 

In this section we describe two further generalizations of our main result. In Section \ref{sec:trunc_bdstates} we study the case where a periodic structure hosting a defect state with $E < 0$ is truncated far from the defect, creating a new bound state whose associated eigenvalue is exponentially close to $E$. In Section \ref{sec:gen_edge} we study the bound state created when a semi-infinite periodic structure hosting an edge state is truncated far from its edge. In remarks \ref{rem:half-line} and \ref{rem:quasi-1d} we discuss potential further generalizations to half-line operators and quasi-one-dimensional operators, respectively.

\subsection{One-dimensional defect states with negative energy $E < 0$ perturb to bound states of the truncated structure} \label{sec:trunc_bdstates}

We have so far assumed that the defect state eigenvalue is positive $E > 0$ throughout. In this section we again consider \eqref{eq:Schro} in the simplest case , i.e. under Assumptions \ref{as:V_assump} and \ref{as:parity}, but assume $H$ has a bound state as in \eqref{eq:bound} but with negative energy $E < 0$ (the generalization when parity symmetry does not hold is straightforward along the same lines as when $E > 0$). Our aim is now to prove that when the structure modeled by $H$ is truncated sufficiently far from $x = 0$, the resulting structure supports a bound state. More specifically, we will prove that the operator $H_{\rm trunc}$ defined by \eqref{eq:H_trunc}-\eqref{eq:V_trunc} has a bound state with energy $E^* < 0$ nearby to $E$ for $M$ sufficiently large (in particular, such that $M > \rho$, where $\rho$ is the support of the defect region). 

Let $u_z(x)$ and $X_1(z), X_2(z)$ be as in \eqref{eq:eq} and \eqref{eq:X}, and then define
\begin{equation} \label{eq:thet_bd}
    \Theta(z) := X_2(z) - i \sqrt{z} X_1(z).
\end{equation}
In contrast to \eqref{eq:theta}, we now assume $z \in \field{C} \setminus [0,\infty)$ and choose the branch of the square root such that $\Im \sqrt{z} > 0$. With these definitions, $E^* < 0$ is a bound state whenever $\Theta(E^*) = 0$.

The counterpart of Theorem \ref{th:main_theorem} in this context is as follows. Note that the theorem requires an additional assumption \eqref{eq:extra}.
\begin{theorem} \label{th:main_theorem_bound}
Let $V(x)$ satisfy Assumptions \ref{as:V_assump} and \ref{as:parity}, let $\Phi$ be a bound state with eigenvalue $E < 0$ as in \eqref{eq:bound}, and let $v_z(x)$ be as in \eqref{eq:eq_v}. Then there exists a $k > 0$ and an $M_0 > \rho \geq 0$ such that for all $M \geq M_0$ such that
\begin{equation} \label{eq:extra}
    v_E'(M) - i \sqrt{E} v_E(M) \neq 0,
\end{equation}
then
\begin{enumerate}
\item $\Theta(z)$ given by \eqref{eq:theta} has a unique root $E^*$ in the ball
\begin{equation}
    \Omega_M = \left\{ z : |z - E| \leq \frac{1}{M^2} e^{- k M} \right\}.
\end{equation}
\item The root $E^*$ is real, and can be precisely characterized as 
\begin{equation} \label{eq:asymp_real}
    E^* = E - \frac{ \Theta(E) }{ \de_z \Theta(E) } + O(e^{- 4 k M}).
\end{equation}
\end{enumerate}
\end{theorem}
Since the proof of Theorem \ref{th:main_theorem_bound} is so similar to that of Theorem \ref{th:main_theorem} we omit it, although we point out the following:
\begin{itemize}
\item The condition \eqref{eq:extra} is also necessary for the proof of Theorem \ref{th:main_theorem} when $E > 0$, but holds automatically because $v_E(x)$ is real for all $x$ and in that case $i \sqrt{E}$ is purely imaginary. When $E < 0$, $i \sqrt{E} = i \sqrt{-|E|} = - \sqrt{|E|}$ is purely real.
\item To see that the root $E^* < 0$ must be real, note that $\frac{\Theta(z)}{\de_z \Theta(z)}$ is real for real and negative $z$. and hence the map $\Psi$ \eqref{eq:psi_z} actually maps $\field{R} \rightarrow \field{R}$ in this case.
\end{itemize}
Corollaries \ref{cor:imag_asymp} and \ref{cor:summary} directly generalize to this context. Since the results and proofs are identical after subtituting $- \sqrt{|E|}$ for $i \sqrt{E}$, we omit them.

We finally remark on the case $E = 0$.
\begin{remark} \label{rem:E=0}
The case $E = 0$ requires a specialized analysis which we do not undertake here. To see why our methods do not generalize, note for example that with $\Theta(z)$ defined by \eqref{eq:theta}, the derivative $\de_z \Theta(E)$ which appears in the definition of $\Psi$ \eqref{eq:psi_z}, and then in the formula for $z^*$ \eqref{eq:asymp}, is not defined when $E = 0$.
\end{remark}

\subsection{Edge states persist under truncation away from the edge} \label{sec:gen_edge}

In recent years, bound states which decay away from the physical edge of a periodic structure known as edge states have attracted significant attention, see e.g. \cite{fefferman_leethorp_weinstein_memoirs,2017FeffermanLee-ThorpWeinstein_2,2018FeffermanWeinstein,2019Drouot_2,2018DrouotFeffermanWeinstein_pre,2019Lee-ThorpWeinsteinZhu,2019DrouotWeinstein}. These works generally assume the existence of a semi-infinite periodic structure adjoining the edge. Without loss of generality we assume the semi-infinite periodic structure extends in the positive $x$ direction, and consider the case where this adjoining semi-infinite periodic structure is truncated at $x = M$ where $M$ is large.

Specifically, we consider the model equation \eqref{eq:Schro} as before, but replace the regularity and translation symmetry assumption (Assumption \ref{as:V_assump}) by the following.
\begin{assumption}[Regularity and translation symmetry of $V$] \label{as:V_assump_2}
We first assume $V$ is smooth: $V(x) \in C^\infty(\field{R})$. Then, we assume that $V$ can be written
\begin{equation}
    V(x) = \begin{cases} V_{\rm per}(x) + V_{\rm def}(x) & x \geq 0 \\ 0 & x < 0 \end{cases},
\end{equation}
where $V_{\rm per}$ is periodic, i.e. $V_{\rm per}(x+1) = V_{\rm per}(x)$ for all $x \geq 0$, and $V_{\rm def}(x)$ is compactly supported, i.e. there exists $\rho \geq 0$ such that 
\begin{equation} \label{eq:rho_again}
    x \geq \rho \implies V_{\rm def}(x) = 0.
\end{equation}
\end{assumption}
Just as in \eqref{eq:bound}, we suppose that $H$ has a bound state $\Phi(x) \in L^2(\field{R})$ with negative energy $E < 0$,  and consider the operator
\begin{equation} \label{eq:H_trunc_again}
    H_{\rm trunc} := D_x^2 + V_{\rm  trunc}(x)
\end{equation}
where
\begin{equation}
    V_{\rm trunc}(x) := \begin{cases} V(x) & - \infty \leq x \leq M \\ 0 & \phantom{-} M < x. \end{cases}
\end{equation}
Following the proof in the case of a regular defect state when parity symmetry does not hold (see Section \ref{sec:gen_no_parity}), and adopting the conventions in Section \ref{sec:trunc_bdstates}, we have that $H_{\rm trunc}$ has a bound state with energy $E^* = E + O(e^{- 2 k M})$, where the first-order correction can be computed explicitly in terms of solutions of the ODEs \eqref{eq:eq_nosymmetry_phi_non-zero}-\eqref{eq:eq_nosymmetry_phiprime_non-zero} as in Corollary \ref{cor:correction_noparity}. Note that the bound state of the truncated structure may not be exponentially small at the truncation, see \cite{2018ThickeWatsonLu}.

\begin{remark}
Note that the operator $H$ \eqref{eq:Schro} with $V$ as in Assumption \ref{as:V_assump_2} cannot have positive energy bound states since solutions of the eigenequation \eqref{eq:bound} with $E > 0$ cannot decay as $x \rightarrow - \infty$.
\end{remark}

We close this section by remarking on one further immediate generalization of our result, and on one potential generalization which does not follow immediately from our methods.

\begin{remark} \label{rem:half-line}
Our results would generalize easily to the setting of Schr\"odinger operators on the half-line. Specifically, we can consider the operator \eqref{eq:Schro} acting on $L^2(0,\infty)$, where $V(x)$ satisfies Assumption \ref{as:V_assump_2}, while imposing a Robin boundary condition at $x = 0$. Assume $H$ has a bound state with eigenvalue $E$. Then the operator $H_{\rm trunc}$, where the potential $V(x)$ is set equal to zero for $x \geq M$, would have a resonance $z^*$ satisfying $z^* - E = O(e^{- 2 k M})$ whenever $E > 0$, or an eigenvalue $E^*$ satisfying $E^* - E = O(e^{- 2 k M})$ whenever $E < 0$.
\end{remark}

\begin{remark} \label{rem:quasi-1d}
Our result could potentially be generalized to the setting of quasi-one-dimensional truncated structures: $d > 1$ dimensional structures which are invariant under translations in $d - 1$ directions. Such a structure can confine waves along $d-1$-dimensional defect regions. When $d = 2$ or $3$, such regions are known as line defects and plane defects, respectively. 

If the structure is truncated in the direction orthogonal to the defect region, we expect that the truncated structure will support a resonance exponentially close to the defect state eigenvalue. Imposing Bloch-periodicity of the wave-function in all directions orthogonal to the direction of truncation, the condition for a resonance reduces to a condition on the wave-function along a quasi-one-dimensional strip. However, unless the strip is \emph{exactly} one-dimensional, our methods (being ODE-based) cannot be directly applied. 

An example of an operator where the reduced problem would be precisely one-dimensional is the operator $- \de_{x}^2 - \de_{y}^2 + V_{per}(x) + V_{def}(x) + W(y)$, where $V_{per}, V_{def},$ and $W$ are real functions, with $V_{per}$ periodic and $V_{def}$ having compact support. By expanding the solution in a basis of eigenfunctions of $- \de_{y}^2 + W(y)$, the reduced problem for a resonance becomes precisely one-dimensional. Since this model seems unrealistic as a model of a structure with a localized defect, we do not consider this case further.
\end{remark}


\section{Conclusions and perspectives}
\label{sec:conjecture}

In this work we have studied the resonances created when periodic structures in one dimension hosting defect states are truncated far from the defect. Our results are restricted to one dimension because our methods rely on ODE theory, but there is no reason to expect that our results will not generalize to higher dimensions. Similarly, although our results are restricted to structures which are periodic away from a defect region because our methods rely on Floquet theory, we also expect our results will generalize to non-periodic systems under appropriate spectral gap assumptions. 

We can formulate a general conjecture regarding the resonances produced when a structure hosting a bound state is truncated away from the bound state maximum. Note that the existence of defect states in dimensions higher than $1$ has been established across various models, see e.g. \cite{1976Simon_2,1986DeiftHempel,1997FigotinKlein,Figotin-Klein:98,Hoefer-Weinstein:11}.

\begin{conjecture*}
In $\field{R}^d$, where $d$ is an arbitrary positive integer, let $H$ denote a Schr\"odinger operator $H := - \Delta + V$, where $V(\vec{x})$ has unbounded support. Assume that $H$ hosts a bound state $\Phi$ whose eigenvalue $E$ is isolated from the rest of the spectrum, i.e. $\sigma(H) \cap B_{\rho}(E) = \{E\}$, where $B_{\rho}(E)$ denotes a ball centered at $E$ with positive radius $\rho > 0$. Let $\vec{\mu} := \inty{\field{R}^d}{}{ \vec{x} | \Phi(\vec{x}) |^2 }{\vec{x}}$ denote the center of mass of the bound state. Define $V_{trunc}(\vec{x})$ by
\begin{equation}
    V_{trunc}(\vec{x}) = \begin{cases} V(\vec{x}) & \vec{x} \in B_M(\vec{\mu}) \\ 0 & \vec{x} \notin B_M(\vec{\mu}), \end{cases}
\end{equation}
where $M > 0$, and define $H_{trunc} := - \Delta + V_{trunc}$. Then, for sufficiently large $M$, and assuming that $\Phi$ is non-zero for some $\vec{x} \notin B_M(\vec{\mu})$ (if not, $\Phi$ is a bound state of $H_{trunc}$), $H_{trunc}$ will have a resonance $z$ in the lower half of the complex plane such that $|z - E|$ is exponentially small in $M$.
\end{conjecture*}

Although the proofs in the present work rely heavily on ODE theory, there is good reason to expect that such a result can be proved. First, exponential decay of bound states in gaps of essential spectrum in arbitrary dimensions follows under general hypotheses from exponential decay of the Green's function (see e.g. \cite{1965Agmon,1973CombesThomas}). Second, in geometric scattering theory there are semiclassical methods developed for proving that classical orbits on manifolds or exterior to obstacles lead to resonances in the absence of any periodic background potential (see, e.g. Chapter 7 ``From quasimodes to resonances'' of Dyatlov-Zworski \cite{DyatlovZworski}).  The methods developed there generalize the fixed point ODE tools used to demonstrate how eigenstates become resonances in the truncated harmonic oscillator that we have applied here in the periodic setting. To study the defects to resonances question in higher dimensional truncated lattice problems, one needs to develop similar tools to the present context, with the existence of a bound state of the infinite structure taking the place of the existence of trapped classical orbits in the semiclassical problem.

\printbibliography

\end{document}